\documentclass[preprint]{article}
\usepackage{lmodern}
\usepackage{graphicx}
\usepackage{caption}
\usepackage{mathtools}
\usepackage{upgreek}
\usepackage{dsfont}
\usepackage{cancel}
\usepackage{todonotes}
\usepackage{stmaryrd} 
\usepackage{ulem}

\usepackage{stackengine}
\usepackage{scalerel}

\usepackage{tikz}
\usepackage{textcomp}
\usepackage[margin=.8in]{geometry}

\usepackage{amsmath}
\usepackage{amssymb}
\usepackage{amsthm}

\usepackage{soul}

\usepackage[hidelinks]{hyperref}

\usepackage{enumerate}
\usepackage[shortlabels]{enumitem} 	
\setlist[itemize]{noitemsep} 			

\newtheorem{corollary}{Corollary}[section]

\newtheorem{theorem}{Theorem}[section]
\newtheorem{lemma}[theorem]{Lemma}
\newtheorem{proposition}[theorem]{Proposition}
\theoremstyle{definition}
\newtheorem{definition}[theorem]{Definition}

\theoremstyle{remark}
\newtheorem{remark}[theorem]{Remark}
\newtheorem{assumption}{Assumption}

\numberwithin{equation}{section}

\newcommand{\T}{\mathbb{T} }
 \newcommand{\R}{\mathbb{R}}
 
  \newcommand{\N}{\mathbb{N}}
 \newcommand{\Z}{\mathbb{Z}}
   \newcommand{\LL}{\mathcal{L}}

\begin{document}

\bibliographystyle{plain}

\title{Almost reducibility of quasiperiodic $SL(2,\mathbb{R})$-cocycles in ultradifferentiable classes
}

\author{Maxime Chatal\footnote{
IMJ-PRG, Universit\'e Paris-Cit\'e, chatal.maxime@gmail.com}\\
and\\
 Claire Chavaudret\footnote{IMJ-PRG, Universit\'e Paris-Cit\'e, chavaudr@math.univ-paris-diderot.fr}
}





\maketitle

\textbf{Abstract:}
Given a quasiperiodic cocycle in $sl(2, \R)$ sufficiently close to a constant, we prove that it is almost-reducible in ultradifferentiable class under an adapted arithmetic condition on the frequency vector. We also give a corollary on the H\" older regularity of the Lyapunov exponent.

\section{Introduction}

\subsection{Presentation of the result}
Let $d\geq 1$ and $\omega = (\omega_1, \dots, \omega_d) \in \R^d$ a rationally independent vector (meaning that no non trivial integer combination of the $(\omega_i)_{ i=1,\dots , d}$ can vanish). We will assume that $\sup_i \vert \omega_i \vert \leq 1$. We will note $\T^d := \R^d/\Z^d$ and $2\T^d := \R^d / 2\Z^d$.
Let $A : \T^d \rightarrow sl(2, \R)$ be in a certain class of continuous matrix-valued functions.
We call quasi-periodic cocycle the solution $X : \T^d \times \R \rightarrow SL(2, \R)$ of the differential linear equation

\begin{equation} \left\{\begin{array}{l} \frac{\mathrm{d}}{\mathrm{d}t}  X^t(\theta) = A(\theta+t\omega)X^t(\theta) \\X^0(\theta) = Id\end{array}\right. \tag{1} \label{eq:cocycle}
\end{equation}

One of the main motivations for studying quasi-periodic cocycles is the study of quasi-periodic Schr\" odinger equations
\[ -y''(t) + q(\theta +t \omega)y(t) = Ey(t) \]
where $q:\mathbb{T}^d\rightarrow \mathbb{R}$ is called the potential, and $E\in\mathbb{R}$ the energy. It gives rise to a cocycle with values in $SL(2,\R)$.
The cocycle is said to be a constant cocycle if $A$ is a constant matrix. A quasi-periodic cocycle as in \eqref{eq:cocycle} is said reducible if it can be conjugated by a quasi-periodic change of variable $Z : \T^d \rightarrow SL(2,\R)$ to a constant cocycle, that is to say, if there exists $B \in sl(2,\R)$ such that, for all $\theta \in 2 \T^d$ :
\[ \partial_\omega Z(\theta) = A(\theta)Z(\theta) - Z(\theta)B\]

In general, it is important to require the change of variables $Z$ to be regular enough. 
Is this paper, we will be interested in the perturbative setting, that is to say, in quasi-periodic cocycles close to a constant :

\begin{equation}\label{perturbative} \left\{\begin{array}{l}\frac{\mathrm{d}}{\mathrm{d}t} X^t(\theta) = (A + F(\theta + t\omega))X^t(\theta) \\X^0(\theta) = Id\end{array}\right. \tag{2}\end{equation}

\noindent
where $A\in sl(2, \R)$ and $F : \T^d \rightarrow sl(2, \R)$ is of ultra-differentiable class and small enough, with a smallness condition depending on $\omega$.

Reducibility is a strong property because it implies that the dynamics will be easily described by the constant equivalent of the system, in particular the Lyapunov exponents, the rotational properties of the solutions, the invariant subbundles etc.
On the counterpart, reducibility results generally require many assumptions. Here we are interested in a weaker property which is almost reducibility. A cocycle like \eqref{perturbative} is said almost-reducible if it can be conjugated by a sequence of quasi-periodic changes of variables to a cocycle of the form
\[ \left\{\begin{array}{l}\frac{\mathrm{d}}{\mathrm{d}t}  X^t(\theta) = (\bar A_n(\theta + t\omega) + \bar F_n(\theta +t\omega)) X^t(\theta) \\X^0(\theta) = Id\end{array}\right. \]
where $\bar{A}_n$ is reducible and $\bar{F}_n$ is arbitrarily small. 

\noindent A quantitative version of almost reducibility, that is, almost reducibility together with estimates on the changes of variable, can have interesting corollaries such as approximate solutions, density of reducible cocycles, regularity of Lyapunov exponents.

\bigskip

\textbf{Ultra-differentiability :} To quantify the regularity of $F \in \mathcal C^{\infty} (\T^d, sl(2, \R))$ and the size of the sequence $(\bar{F}_n)$ above, we introduce the weight function $\Lambda : [0, +\infty[ \rightarrow [0, + \infty[$ which we will assume to be increasing and differentiable. Expanding $F$ in Fourier series $F(\theta) = \sum_{k\in \Z^d} \hat F(k)e^{2i\pi \langle k, \theta \rangle}$, we will say that $F$ is $\Lambda$-ultra-differentiable if there exists $r > 0$ such that
\[ \vert F \vert_{r} = \vert F \vert_{\Lambda, r} := \sum_{k \in \Z^d} \Vert \hat F(k) \Vert e^{2\pi \Lambda (\vert k \vert)r} < \infty \]
where $\vert k\vert$ is the sum of the absolute values of the components of $k$, and we will denote $F\in U_{r}(\T^d,  sl(2, \R)) = U_{\Lambda, r}(\T^d,  sl(2, \R))$. To make this space a Banach algebra, we will require $\Lambda$ to be subadditive :
\[ \Lambda (x+y) \leq \Lambda (x) + \Lambda (y), \quad \forall x,y \geq 0\]


\noindent
If $\Lambda \equiv id$, it is the analytic case.

\bigskip
\begin{remark}
The standard definition of ultra-differentiable functions involves Denjoy-Carleman sequences, that is, real sequences satisfying certain conditions which act as bounds on the successive derivatives of a given function. However, the above definition, introduced by Braun-Meise-Taylor (\cite{BMT}), can be linked to Denjoy-Carleman classes (see \cite{Ra21}, Theorem 11.6). Since Fourier series appear naturally in the problem considered here, we chose to use Braun-Meise-Taylor classes as a starting point.
\end{remark}

\bigskip

\textbf{Non-resonance condition on the frequency:} An often studied situation is the case where the frequency vector $\omega$ if Diophantine (which we denote by $\omega  \in DC(\kappa, \tau)$), for some $0 < \kappa <1$ and $\tau \geq \max(1, d-1)$ :
\[ \vert \langle k, \omega \rangle \vert \geq \frac{\kappa}{\vert k \vert^\tau}, \quad \forall k \in \Z^d \backslash \{0 \}\]
where $\langle \cdot, \cdot\rangle$ is the standard Euclidean inner product. It was proved by Eliasson \cite{3} that in the analytic case, if $\omega \in DC(\kappa, \tau)$, and $F$ is sufficiently small, Equation \eqref{perturbative} is almost reducible. This result was improved by Chavaudret \cite{1} who proved that the convergence occurs on analyticity strips of fixed width (whereas Eliasson's theorem gave the convergence on strips of width going to zero).

One of the aims of the present paper is to weaken this arithmetic condition by introducing the approximating function
\[ \Psi : [0, + \infty[ \rightarrow [0, +\infty[ \]
with $\Psi \geq id$ (which is not restrictive since it is satisfied by the diophantine condition). We will assume $\Psi$ to be increasing, differentiable and satisfying, for all $x, y \in [1, +\infty[$, \[\Psi(x+y) \geq \Psi(x)+\Psi(y)\] 

\noindent thus for all $n\in\mathbb{N}$, and for all $x\geq 1$, $\Psi(nx)\geq n\Psi(x)$.
In our problem, $\omega$ will satisfy the following arithmetic condition for some $\kappa \in ]0, 1[$ :
\[ \vert \langle k, \omega \rangle \vert \geq \frac{\kappa}{\Psi(\vert k \vert)}, \quad \forall k \in \Z^d \backslash \{0 \}\] 
(notice that the case $\Psi(.) = \vert . \vert^\tau$, is the Diophantine case).

We will require the following condition:
\[\lim_{t \rightarrow +\infty}\frac{\log\Psi(t)}{\Lambda(t)} = 0.\]
and
$$\int_0^\infty \frac{\Lambda'(t)\ln \Psi(t)}{\Lambda(t)^2}dt<+\infty$$

This condition, known as the $\Lambda$-Brjuno-R\" ussmann condition, will be denoted by $\omega \in BR(\kappa)$.
This coincides with the well-known Brjuno condition if $\Lambda$ is the identity.

If $A$ is elliptic, an almost reducibility theorem was given in \cite{2}. 

The purpose of this article is to show the following theorem:

\begin{theorem}\label{maintheorem}
Let $r_0 >0$, $A_0\in sl(2, \R)$ and $F_0 \in U_{r_0}(\T^d, sl(2, \R))$. Then, there exists $\varepsilon_0$ depending only on $A_0,\kappa,\Lambda,\Psi,r_0$
 such that, if
\[ \vert F_0 \vert_{r_0} \leq \varepsilon_0\]
then for all $\varepsilon \leq \varepsilon_0$, 
there exist 

\begin{itemize}
\item $r_\varepsilon >0$, $\zeta \in ]0,\frac{1}{8}[$,
\item $Z_\varepsilon \in U_{r_\varepsilon}( \T^d, SL(2, \R))$,
\item $A_\varepsilon \in sl(2, \R)$,
\item $\bar A_\varepsilon, \bar F_\varepsilon \in U_{r_\varepsilon}(\T^d, sl(2, \R))$,
\item $\psi_\varepsilon \in U_{r_{\varepsilon}}(2\T^d, SL(2, \R))$
\end{itemize}
such that
\begin{enumerate}
\item $\bar A_\varepsilon$ is reducible to $A_\varepsilon$ by $\psi_\varepsilon$, with $\vert \psi_\varepsilon  \vert_{r_\varepsilon} \leq \varepsilon^{-\frac{1}{2}\zeta}$,
\item $\vert \bar F_\varepsilon \vert_{r_\varepsilon} \leq \varepsilon$,
\item $\lim_{\varepsilon \rightarrow 0} r_\varepsilon >0$,
\item for all $\theta \in \T^d$,
\[ \partial_\omega Z_\varepsilon (\theta) = (A_0 + F_0(\theta)) Z_\varepsilon (\theta) - Z_\varepsilon (\theta)(\bar A_\varepsilon (\theta) + \bar F_\varepsilon (\theta))\]
\item \[ \vert Z_\varepsilon ^{\pm 1} - Id \vert_{r_\varepsilon} \leq \varepsilon_0^{\frac{9}{10}} \]
\end{enumerate}

Moreover, either $\Psi_\epsilon$ becomes constant as $\epsilon\rightarrow 0$, or there exist arbitrarily small $\varepsilon$ such that 
$||A_\varepsilon||\leq \kappa \varepsilon^\zeta$.
\end{theorem}

This theorem states almost reducibility in an ultradifferentiable class with the same weight function as that of the initial system, but with a smaller parameter $r_\varepsilon$. Notice however that the parameter $r_\varepsilon$ does not shrink to $0$. In order to achieve this, the resonance cancellation technique is similar to the one in \cite{1}. Notice that, for topological reasons, a period doubling is necessary in order to preserve the real structure. This phenomenon was already observed in \cite{C1}.

%
%
%
\subsection{Discussion}

Almost reducibility in itself is an interesting property of quasi-periodic cocycles, in particular the quantitative version. A perturbative almost reducibility result in arbitrary dimension of space was given in \cite{3}, in the analytic framework, under a diophantine condition on the frequency vector. A quantitative version of perturbative almost reducibility, in a similar framework, was then proved in \cite{1}. Using a technique developed in \cite{AFK}, in \cite{HY}, Hou and You managed to remove the diophantine assumption on the frequency vector, in case it is 2-dimensional ($d=2$); thus the result became non perturbative if not global (see also \cite{Puig} for a non perturbative reducibility result). It is yet unknown whether arithmetical conditions can be removed in case $d>2$. However in the analytic case, for any number of frequencies, it is known that the diophantine condition is not optimal for reducibility results and can be replaced by the Brjuno-R\"ussmann condition (see \cite{LD06}, \cite{CM12}). Here we give an almost reducibility result in which the arithmetical condition coincides with the Brjuno-R\"ussmann condition in the analytic case. 

\bigskip
Concerning the functional framework, a few results were known in the Gevrey class. The reference \cite{1} contains almost reducibility in the Gevrey class as well, under a diophantine condition. The reference \cite{HP} gives a result on rigidity of reducibility in the Gevrey class, under a diophantine condition (see also
\cite{LDG19} 
on Gevrey flows). But a simultaneous extension of Eliasson's reducibility result in \cite{E92} to more general ultradifferentiable classes and to a weaker arithmetical condition, which is linked to the considered class of functions, was given in \cite{2} (see also \cite{BF} for a result in a hamiltonian setting). Here, we obtain this generalization for almost reducibility, the proof of which is more technical. The link between the arithmetical condition and the functional setting is similar to the one in \cite{2}, and also coincides with the Brjuno-R\"ussmann condition in the analytic case.

\subsection{Comments on the proof} 

The proof of the main result relies on the well-known KAM algorithm: a step of the algorithm will reduce the size of the perturbation to a power of it, by means of a change of variables which might be far from identity (if resonances have to be cancelled), but is still controlled by a small negative power of the size of the perturbation. The order at which one removes resonances to avoid small divisors has to be suitably chosen in order to decrease the perturbation sufficiently while having a sufficient control on the change of variables. One also has to shrink the parameter of the ultradifferentiable class at every step, and in order to have a strong almost reducibility result (i.e a sequence of parameters not shrinking to $0$), the $\Lambda$-Brjuno-R\"ussmann condition comes naturally. 

\noindent If resonances are cancelled only finitely many times, then the change of variables remains close to identity at every step afterwards, which gives reducibility. Otherwise, the constant part of the system itself becomes small.

\bigskip
The main theorem, which is Theorem \ref{theoreme} below, is proved by iterating arbitrarily many times the Lemma \ref{352} below; Lemma \ref{352} gives a conjugation between to systems $\bar{A}+\bar{F}$ and $\bar{A}'+\bar{F}'$, where both $\bar{A}$ and $\bar{A}'$ are reducible maps and $\bar{F}'$ is smaller than $\bar{F}$, with a controlled loss of regularity. 

\bigskip
The proof of Lemma \ref{352} can be sketched by the following diagram:

$$\bar{A}+\bar{F} \underset{Lemma\ \ref{351}}{\overset{\psi}{\longrightarrow} A+F\overset{\psi^{-1}\Phi^{-1}e^{X}\Phi}{\longrightarrow} }\bar{A}_1+\bar{F}_1
\overset{\Phi \psi}{\longrightarrow} A_1+F_1\underset{Lemma\ \ref{344}}{\overset{e^{X_1}}{\longrightarrow}}\dots \underset{Lemma\ \ref{344}}{\overset{e^{X_{l-1}}}{\longrightarrow}} A_l+F_l\overset{\psi^{-1}\Phi^{-1}}{\longrightarrow} \bar{A}_l+\bar{F}_l=\bar{A}'+\bar{F}'$$

where $A,A_1,\dots, A_l$ are constant matrices, $\bar{A},\bar{A}_1,\dots, \bar{A}_l$ are reducible, and $\bar{F},F,\bar{F}_i,F_i$ are small.

\bigskip
No non resonance condition is required on $A$, making it necessary to construct the change of variables $\Phi$ which will remove resonances, but may be far from the identity (Lemma \ref{renormalization}). However, once this is done, the matrices $A_1,\dots, A_{l-1}$ remain non resonant enough in order to reduce the perturbation a lot without having to remove resonances again.

\bigskip
The superscripts on the arrows refer to the changes of variables. The changes of variables with an exponential expression are close to the identity, therefore the total conjugation, from $\bar{A}+\bar{F}$ to $\bar{A}'+\bar{F}'$, is close to identity, which makes it possible to obtain the density of reducible systems in the neighbourhood of a constant.

\section{Notations}

The notation $E(x)$ will refer to the integer part of a number $x$.

If $F \in L^2(2\T^d)$ and $N \in \N$, the truncation of $F$ at order $N$ (denoted $F^N$) is the function we obtain by cutting the Fourier series of $F$ :
\[ F^N(\theta) = \sum_{\vert m \vert \leq N}\hat F(m)e^{2i\pi \langle k, \theta \rangle} \]

In order to simplify the notation throughout this paper, we will write $\Psi(\cdot)$ for $\Psi(\vert \cdot \vert)$, and $\Lambda(\cdot)$ for $\Lambda(\vert \cdot \vert)$.

We will denote by $||\cdot ||$ the norm of the greatest coefficient for matrices.

\section{Decompositions, triviality}

We take the following definitions from \cite{1}, describing decompositions of 
$\R^2$ and triviality, which will avoid to double the period more than once.

\begin{definition}[Decomposition]
\begin{itemize}[label=$\bullet$]
\item If $A \in sl(2,\R)$ has distinct eigenvalues, we call \textit{$A$-decomposition} a decomposition of $\R^2$ as the direct sum of two eigenspaces of $A$.
If $L$ is an eigenspace of $A$, we write $\sigma(A_{\vert L})$ the spectrum of the restriction of $A$ to the subspace $L$.
We shall denote by $\LL_{A}$ the decomposition of $\R^2$ into two distinct eigenspaces of $A$, if the related eigenvalues are distinct.
\item If $\R^2=L_1\bigoplus L_2$,
for all $u \in \R^2$, there exists a unique decomposition $u = u_1+u_2, u_1\in L_1,u_2\in L_2$.
For $i=1,2$, we call \textit{projection on $L_i$ with respect to $\mathcal{L}=\{L_1,L_2\}$}, and we write $P_{L_i}^\LL$ the map defined by $P_{L_i}^\LL u = u_i$.
\end{itemize}
\end{definition}

Recall the following lemma on estimate of the projection (see \cite{3}) :
\begin{lemma}[\cite{3}]\label{311} Let $\kappa ' >0$ and $A \in sl(2, \R)$ with $\kappa'$-separated eigenvalues. 
There exists a constant $C_0 \geq 1$ such that, for any subspace $L \in \LL=\LL_{A}$,
\[ \Vert P_L^\LL \Vert \leq C_0\big( \frac{1}{\kappa '}\big)^{6}\] 
\end{lemma}

\begin{remark}
The estimate given in \cite{3} is more general since it also concerns matrices $A$ with a nilpotent part. Here the setting in $sl(2,\mathbb{R} )$ makes the estimate a little better.

\end{remark}

\begin{definition}[Triviality]
Let $\LL=\{L_1,L_2\}$ such that $L_1\bigoplus L_2=\R^2$. 
We say that a function $\psi \in \mathcal{C}^0(2\T^d,SL(2\R))$ is trivial with respect to $\LL$ if there exists $m\in \frac{1}{2}\Z^d$ such that for all $\theta \in 2\T^d$,
\[ \psi(\theta) = e^{2i\pi \langle m, \theta \rangle}P_{L_1}^\LL +e^{-2i\pi \langle m, \theta \rangle}P_{L_2}^\LL\]

\noindent If $|m|\leq N$, we say that $\psi$ is trivial of order $N$.
\end{definition}

\begin{remark}\label{period-properties}
\begin{itemize}
\item If $\psi_1, \psi_2 : 2\T^d \rightarrow SL(2, \R)$ are trivial with respect to $\LL$, then the product $\psi_1 \psi_2$ is also trivial with respect to $\LL$ since, for all $L \neq L', P_L^{\LL}P_{L'}^{\LL} = 0$.
\item If $\psi$ is trivial with respect to a a decomposition $\LL$ of $\R^2$, then for all $G \in \mathcal{C}^0(\T^d, sl(2, \R))$, we have $\psi G \psi^{-1} \in \mathcal{C}^0(\T^d, sl(2, \R))$. 
Indeed,
notice that if $\psi =  e^{2i\pi \langle m, \cdot \rangle}P_{L_1}^{\LL} +e^{-2i\pi \langle m, \cdot \rangle}P_{L_2}^{\LL}$ for some $m\in \frac{1}{2}\mathbb{Z}^d$, then \[ \psi^{-1} =  e^{-2i\pi \langle m, \cdot \rangle}P_{L_1}^{\LL} +e^{2i\pi \langle m, \cdot \rangle}P_{L_2}^{\LL} \] (it's a simple calculus to check that with this expression, $\psi \psi^{-1} = \psi^{-1}\psi \equiv I$.) Then, 

\begin{equation*}\begin{split}
\psi G \psi^{-1} &= (e^{2i\pi \langle m, \cdot \rangle}P_{L_1}^{\LL} +e^{-2i\pi \langle m, \cdot \rangle}P_{L_2}^{\LL})G(e^{-2i\pi \langle m, \cdot \rangle}P_{L_1}^{\LL} +e^{2i\pi \langle m, \cdot \rangle}P_{L_2}^{\LL})  \\
		         &= P_{L_1}^{\LL}GP_{L_1}^{\LL} + P_{L_2}^{\LL}GP_{L_2}^{\LL} + e^{2i\pi\langle 2m, \cdot\rangle}P_{L_1}^{\LL}GP_{L_2}^{\LL}+e^{-2i\pi\langle 2m, \cdot\rangle}P_{L_2}^{\LL}GP_{L_1}^{\LL}
\end{split}\end{equation*}
which is well defined continuously on $\T^d$.
Hence the function $\psi$ will avoid a period doubling.
\end{itemize}
\end{remark}

\section{Choice of parameters}
In this section, we define all the constants and parameters used in this paper.
\[ \left\{\begin{array}{c}
\delta= 100000\\
\zeta = \frac{1}{1728}\end{array}\right. \]

Let for all $r, \varepsilon >0$,

\[ N(r,\varepsilon) = \Lambda^{-1}\big(\frac{50     \vert \log \varepsilon \vert}{\pi r}\big) \]
\[ R(r,\varepsilon) = \frac{1}{3N(r,\varepsilon)}\Psi^{-1}(\varepsilon^{-\zeta})\]
\[ \kappa''(\varepsilon) =  \kappa \varepsilon^{\zeta}\]
\[ r'(r,\varepsilon) = r-\frac{50 \delta \vert \log \varepsilon \vert}{\pi \Lambda(R(r,\varepsilon)N(r,\varepsilon))}\]

\section{Smallness of the perturbation}\label{smallness-assumption}

Let $r_0>0,A_0\in sl(2,\R),F_0\in U_{r_0}(\mathbb{T}^d, sl(2,\mathbb{R}))$.

\begin{assumption}\label{assumption-2bis}
The functions $\Lambda,\Psi$ satisfy 

$$\lim_{t\rightarrow +\infty} \frac{\ln\Psi(t)}{\Lambda(t)} = 0$$

and $\varepsilon_0$ is small enough as to satisfy conditions of lemma \ref{smallness-epsilon} below and:

\[  \frac{150\delta\vert \log \varepsilon_0 \vert}{\pi \Lambda\circ \Psi^{-1}(\epsilon_0^{-\zeta})} +\frac{150\delta}{\pi  \zeta \log(2\delta)}
\int_{\Psi^{-1}(\varepsilon_0^{-\zeta})}^{+\infty}\frac{\Lambda'(t)\ln \Psi(t)}{\Lambda(t)^2} dt <r_0
\]

\end{assumption}

These conditions depend on $r_0,\Lambda,\Psi$.

\bigskip
 We shall define the following sequences of parameters used throughout the iteration:
\[ \varepsilon_k := \varepsilon_0^{(2\delta)^k}\]
\[ \Lambda(N_k) := \Lambda(N(r_k,\varepsilon_k)) = \frac{50     \vert \log \varepsilon_k \vert}{\pi r_k} \]
\[ R_k := R(r_k,\varepsilon_k) = \frac{1}{3N(r_k,\varepsilon_k)}\Psi^{-1}(\varepsilon_k^{-\zeta}) \]
and
\[ r_k :=  r_0- \sum_{i=0}^{k-1}\frac{50 \delta \vert \log \varepsilon_i \vert}{\pi \Lambda(R(r_i,\varepsilon_i) N(r_i,\varepsilon_i))}\]

\begin{lemma}\label{rk-limite-positive}
Under either the Assumption \ref{assumption-2bis}, the sequence $r_k$ converges to a positive limit.
\end{lemma}

\begin{proof}
Notice that, for all $k$, and since $\Lambda$ is subadditive, 
\[\Lambda(R_kN_k) \geq \frac{1}{3}\Lambda(3R_kN_k) \Rightarrow \frac{1}{\Lambda(R_kN_k)} \leq \frac{3}{\Lambda(3R_kN_k)}\]
Then
\begin{align*}
\sum_{k\geq 0} \frac{50 \delta \vert \log \varepsilon_k \vert}{\pi \Lambda(R_k  N_k)}  & \leq \sum_{k\geq 0} \frac{150 \delta \vert \log \varepsilon_k \vert}{\pi \Lambda(3R_k  N_k)} \\
& \leq  \frac{150\delta}{\pi }\sum_{k \geq 0} \frac{(2\delta)^k\vert \log \varepsilon_0 \vert}{ \Lambda(3R_k  N_k)}  \\
        &  \leq \frac{150\delta}{\pi }\sum_{k \geq 0} \frac{(2\delta)^k\vert \log \varepsilon_0 \vert}{\Lambda (\Psi^{-1}(\varepsilon_k^{-\zeta}))} \\
        &  \leq  \frac{150\delta\vert \log \varepsilon_0 \vert}{\pi \Lambda\circ \Psi^{-1}(\epsilon_0^{-\zeta})} 
        +\frac{150\delta\vert \log \varepsilon_0 \vert}{\pi }
        \int_0^{+\infty} \frac{(2\delta)^x}{\Lambda (\Psi^{-1}(\varepsilon_0^{-\zeta(2\delta)^x}))}dx \\
\intertext{With the change of variable $t := (2\delta)^x$}
        & \leq \frac{150\delta\vert \log \varepsilon_0 \vert}{\pi \Lambda\circ \Psi^{-1}(\epsilon_0^{-\zeta})} +\frac{150\delta\vert \log \varepsilon_0 \vert}{\pi } \int_{1}^{+\infty} \frac{t}{\Lambda (\Psi^{-1}(\varepsilon_0^{-\zeta t}))} \frac{1}{t \log (2\delta)}dt \\
            & \leq \frac{150\delta\vert \log \varepsilon_0 \vert}{\pi \Lambda\circ \Psi^{-1}(\epsilon_0^{-\zeta})} + \frac{150\delta\vert \log \varepsilon_0 \vert}{\pi  \log (2\delta)} \int_{1}^{+\infty} \frac{1}{\Lambda (\Psi^{-1}(\varepsilon_0^{-\zeta t}))} dt \\
\intertext{With the change of variable $v := \Psi^{-1}(\varepsilon_0^{-\zeta t})$}
        & \leq \frac{150\delta\vert \log \varepsilon_0 \vert}{\pi \Lambda\circ \Psi^{-1}(\epsilon_0^{-\zeta})} +\frac{150\delta\vert \log \varepsilon_0 \vert}{\pi  \log (2\delta)} \int_{\Psi^{-1}(\varepsilon_0^{-
        \zeta})}^{+\infty} \frac{1}{\Lambda(v)}
       \cdot  \frac{  \Psi'(v)}{ -\zeta \log \varepsilon_0 \Psi(v)}dv \\
        & \leq \frac{150\delta\vert \log \varepsilon_0 \vert}{\pi \Lambda\circ \Psi^{-1}(\epsilon_0^{-\zeta})} +\frac{150\delta}{\pi \zeta  \log(2\delta)}\int_{\Psi^{-1}(\varepsilon_0^{-
        \zeta})}^{+\infty} \frac{\Psi'(v)}{\Lambda(v)\Psi(v)}dv
\end{align*}

\noindent 
After integrating by parts, 

\begin{equation}\begin{split} \sum_{k\geq 0} \frac{50 \delta \vert \log \varepsilon_k \vert}{\pi \Lambda( R_k  N_k) }& \leq 
\frac{150\delta\vert \log \varepsilon_0 \vert}{\pi \Lambda\circ \Psi^{-1}(\epsilon_0^{-\zeta})} +
\frac{150\delta}{\pi \zeta  \log(2\delta)}
[
-\frac{\log(\varepsilon_0^{-\zeta })}{\Lambda(\Psi^{-1}(\varepsilon_0^{-\zeta}))}\\
&+\int_{\Psi^{-1}(\varepsilon_0^{-\zeta })}^{+\infty}  \frac{\Lambda'(v)\log \Psi(v)}{\Lambda(v)^2}dv] \\
&\leq \frac{150\delta\vert \log \varepsilon_0 \vert}{\pi \Lambda\circ \Psi^{-1}(\epsilon_0^{-\zeta})} +\frac{150\delta}{\pi \zeta  \log(2\delta)}\int_{\Psi^{-1}(\varepsilon_0^{-\zeta })}^{+\infty}  \frac{\Lambda'(v)\log \Psi(v)}{\Lambda(v)^2}dv
\end{split}\end{equation}

provided $\displaystyle \lim_{v \rightarrow + \infty}\frac{\log \Psi(v)}{\Lambda(v)} = 0$, thus the assumption \ref{assumption-2bis} implies that $(r_k)$ converges to a positive limit. 
\end{proof}

\textit{Remark :} We naturally find the Bruno-Rüssmann condition with respect to weight function $\Lambda$, which is the convergence of $\displaystyle \int  \frac{\Lambda'(v)\log\Psi(v)}{\Lambda(v)^{2}}dv $.

\section{Elimination of resonances}

Given a matrix $A$, a useful technique in the KAM iteration will be to remove the resonances in the spectrum of the matrix $A$. To characterize the non-resonance of $z \in \mathbb{C}$ (depending on $\omega$, a constant $\kappa'>0$ and on an order $N \in \N$) we will write $z \in BR_{\omega}^N(\kappa')$ if and only if :
\[ \forall k \in \mathbb{Z}^d \backslash \{0\}, \quad 0 < \vert k \vert \leq N \Rightarrow \vert z - 2i \pi \langle k, \omega \rangle \vert \geq \frac{\kappa'}{\Psi (k)} \]

\noindent
\begin{definition}
We will say that $A$ has $BR_\omega^N(\kappa')$ spectrum if
\[ \sigma(A)=\{\alpha, \alpha'\} \Rightarrow \alpha - \alpha' \in BR_{\omega}^N(\kappa ') \]
In particular, if the eigenvalues of $A$ are $i\alpha$ and $-i\alpha$ with $\alpha \in \R$, $A$ has if $BR_\omega^N(\kappa')$ spectrum if
\[  2i\alpha \in BR_{\omega}^N(\kappa ') \]
\end{definition}

\begin{remark}
If $A$ has real eigenvalues $\alpha, \alpha'$, then for all $N\in \N$, $A$ has $BR_\omega^N(\kappa)$-spectrum because $|\alpha-\alpha'-2i\pi \langle m,\omega\rangle |\geq |2i\pi \langle m,\omega\rangle |\geq \frac{\kappa}{\Psi(m)}$.

\end{remark}

\begin{lemma} \label{resonances1}
Let $\alpha \in \mathbb{R}$, $\tilde{N}\in\mathbb{N^*}$ and $\kappa' = \displaystyle \frac{\kappa}{\Psi(3\tilde{N})}$. There exists $m \in \frac{1}{2}\mathbb{Z}^d$, 
$\vert m \vert\leq  \frac{1}{2}\tilde{N}$ such that, if we denote $\alpha' = \alpha - 2\pi \langle m,\omega\rangle$, then $2i\alpha' \in BR_{\omega}^{\tilde{N} }(\kappa')$ and if $m\neq 0$ then $\vert \alpha'\vert \leq \frac{\kappa'}{2}$.

\end{lemma}

\begin{proof}
We want to remove the resonances between $i\alpha$ and $-i\alpha$. If $2i\alpha \in BR_{\omega}^{\tilde N}(\kappa ')$, let $m=0$ and we are done. Otherwise, if there exists $ m' \in \frac{1}{2}\Z^d$ with $\vert  m' \vert \leq \tilde N$ such that 
\[ \vert 2\alpha - 2\pi \langle m', \omega \rangle  \vert < \frac{\kappa}{\Psi(m')}\]
then let $m = \frac{m'}{2}$ and $2\alpha ' = \alpha - 2\pi \langle m, \omega \rangle$. It's a simple calculus to check that, in this case, $\vert 2\alpha' \vert \leq \kappa'$, hence $\alpha'\leq \frac{\kappa'}{2}$. Now, for all $k\in \frac{1}{2}\Z^d$, $k\leq \tilde N$,
\[ \vert 2 i \alpha' - 2i\pi \langle k, \omega \rangle \vert  \geq \frac{\kappa}{\Psi(k)} - \kappa' \geq \frac{\kappa'}{\Psi(k)}\]

Then $2i\alpha' \in BR_{\omega}^{\tilde N}(\kappa')$.
\end{proof}

\begin{lemma} \label{resonnances2}
Let $\alpha\in\mathbb{R}$. For all $R\in\mathbb{R},N \in \mathbb{N}$, $N\geq 1, R \geq 2$, there exists $m \in \frac{1}{2}\mathbb{Z}^d$, $\vert m\vert \leq \frac{1}{2}N$  such that, if we denote 
$\displaystyle \kappa '' = \frac{\kappa}{ \Psi(3RN)}$ and $\alpha' = \alpha-2\pi\langle m,\omega\rangle$, then $2i\alpha' \in BR_{\omega}^{RN} (\kappa'')$ and if $m\neq 0$ then $\vert \alpha' \vert \leq \frac{\kappa''}{2}$.

\end{lemma}

\begin{proof} If $\alpha  \in BR_{\omega}^{RN} (\kappa'')$, then $m=0$. Otherwise, apply the previous Lemma with $\tilde{N}=N$ and $\kappa'=\kappa''$ to obtain $\vert m\vert \leq \frac{1}{2}N$ such that $\vert \alpha - \langle m,\omega\rangle\vert \leq \frac{\kappa''}{2}$. Therefore for all $0\leq \vert k\vert \leq RN$, 

$$\vert 2\alpha' - 2\pi\langle k,\omega\rangle\vert \geq \vert 2\pi\langle k,\omega\rangle \vert  - \kappa''\geq \frac{\kappa}{\Psi( k )}-\kappa''\geq \frac{\kappa''}{\Psi(k)}$$
and then $2i\alpha' = 2i\alpha - 2\pi\langle m,\omega \rangle \in BR_{\omega}^{RN}(\kappa'')$.
\end{proof}

\section{Renormalization}
We want to define a map $\Phi$ which conjugates $A$ to a matrix with $BR_\omega^{R  N}(\kappa'')$ spectrum.

\begin{lemma}\label{renormalization} Let $A \in sl(2, \mathbb{R}), R\geq 2, N \in \mathbb{N}\setminus\{0\}$. If $\kappa '' = \displaystyle \frac{\kappa}{\Psi(3RN)}$ and $A$ has $\kappa''$-separated eigenvalues, then there exists a map $\Phi \in \mathcal{C}^0(2\T^d, SL(2, \mathbb{R}))$ which is trivial with respect to $\mathcal{L}_A$ (the decomposition into eigenspaces of $A$), and a constant $C_0\geq 1$ such that,
\begin{enumerate}
\item \label{item1} $\qquad$ For all $r' > 0$, \[  \vert \Phi ^{ \pm1} \vert_{r'}  \leq  2 C_0e^{2\pi \Lambda(\frac{N}{2})r'} \big(  \frac{1 }{\kappa '' }\big)^{6}\]
\item \label{item2} $\qquad$ If $\tilde{A}$ is defined by the following condition: for all $\theta \in 2\T^d$, 

$$ \partial_\omega \Phi(\theta) = A \Phi(\theta) - \Phi(\theta)\tilde{A}$$

\noindent (note that $\tilde A$ actually does not depend on $\theta$), then $\Vert \tilde{A} - A \Vert \leq \pi N$ and $\tilde{A}$ has $BR_\omega ^{R N}(\kappa '')$ spectrum.
\item \label{item3} For any function $G\in \mathcal{C}^0(\T^d, sl(2,\R))$, we have $\Phi G \Phi^{-1} \in  \mathcal{C}^0(\T^d, sl(2,\R))$.
\item \label{item4} If $\tilde{A}\neq A$ then $\Vert \tilde{A}\Vert \leq \frac{1}{2}\kappa''$.
\end{enumerate}
\end{lemma}

\begin{proof} Let $m$ given by Lemma \ref{resonnances2} with $\alpha$ the imaginary part of an eigenvalue of matrix $A$. 
If the eigenvalues of $A$ are in $\R$, then $m=0$ and $\Phi\equiv I$. 

Otherwise, let $L_1$ be the invariant subspace associated to $i\alpha$, $L_2$ associated to $-i\alpha$ and let for all $\theta \in 2 \T^d$, 

\[ \Phi(\theta) = e^{2i\pi \langle m, \theta \rangle}P_{L_1}^{\LL_A} +e^{-2i\pi \langle m, \theta \rangle}P_{L_2}^{\LL_A} \]

For all $\theta$, since the eigenvalues of $\Phi(\theta)$ are complex conjugate, we have $\Phi(\theta) \in SL(2, \R)$, and from Lemma \ref{resonnances2}, $\tilde A$ has $BR_\omega^{R N}(\kappa '')$ spectrum (since the eigenvalues of $\tilde A$ are  $\pm i \tilde \alpha$ obtained from lemma \ref{resonnances2} where $\pm i\alpha$ are the eigenvalues of $A$). Moreover, the spectrum of $\tilde A - A$ is $\{ \pm 2i\pi \langle m, \omega \rangle \}$ and $\vert  2i\pi \langle m, \omega \rangle \vert \leq \pi N$ (remind that $\vert m \vert \leq \frac{1}{2}N$, and that we supposed $\vert \omega \vert \leq 1$) whence \ref{item2}. 
Moreover,  because $\vert m \vert \leq \frac{1}{2}  N$, and from Lemma \ref{311},
\[ \vert \Phi \vert_{r'} \leq ( \Vert P_{L_1} \Vert + \Vert P_{L_2} \Vert) e^{2\pi \Lambda(\frac{N}{2})r'} 
\leq 2C_0\big( \frac{1 }{\kappa ''}\big)^{6} e^{2\pi \Lambda(\frac{N}{2})r'} \]
whence \ref{item1}.
The property \ref{item3} follows from the triviality of $\Phi$ (see the remark \ref{period-properties}).

\bigskip
For the estimate in \ref{item4}, notice that if $A \neq \tilde A$, that is to say if the spectrum of $A$ was resonant, $\Phi \not\equiv I$ conjugates $A$ to $\tilde A$. In particular, from lemma \ref{resonnances2}, the two eigenvalues $i \tilde \alpha$ and $- i\tilde \alpha$ of $\tilde A$ (which are the eigenvalues of $A$ translated by $2i\pi\langle m, \omega \rangle$) satisfy $\vert i\tilde \alpha - (-i \tilde \alpha) \vert \leq \kappa''$ and then $\Vert \tilde A \Vert \leq \frac{1}{2}\kappa''$. \end{proof}
\begin{definition}
A function $\Phi$ satisfying conclusions of lemma \ref{renormalization} will be called \textit{renormalization of $A$ of order $R,N$}. Here the resonance is removed up to order $RN$ whereas 
the estimate involves an exponential of $\Lambda(\frac{N}{2})$ and $\Psi(3RN)$.
\end{definition}


\section{Cohomological equation}

In order to define a change of variables which will reduce the norm of the perturbation, we will first solve a linearized equation, which has the form:

\begin{equation}
\forall \theta \in \T^d, \partial_\omega \tilde X(\theta) = [\tilde A, \tilde X( \theta)] + \tilde F^N - \hat {\tilde F}(0), \quad \hat {\tilde{X}}(0) = 0 
\end{equation}

Here $\tilde{A}\in sl(2,\R)$, therefore either it has real non zero eigenvalues, or it is the zero matrix, or it is nilpotent, or it has two eigenvalues $i\alpha, -i\alpha,\alpha\in \R^*$. 
Only in the latter case can $\tilde{A}$ be resonant.

\bigskip
Assume the eigenvalues are different (so, either they are distinct reals or they are complex conjugates).
Let $L_1,L_2$ be the eigenspaces.
For all $L,L'\in \{L_1,L_2\}$, define the following operator:

$$\mathcal{A}_{L,L'} : gl(2,\R) \rightarrow gl(2, \R), M \mapsto \mathcal{A}_{L,L'} M := \tilde A P_LM - MP_{L'}\tilde A$$

It will be necessary to compute the spectrum of every $\mathcal{A}_{L,L'}$ to estimate the solution of the linearized equation. This is done in the following lemma:

\begin{lemma}\label{spectrum-ALL'}
Let $L,L'\in \{L_1,L_2\}$, $\beta$ the eigenvalue associated to $L$ and $\gamma$ the eigenvalue associated to $L'$. The spectrum of $\mathcal{A}_{L,L'}$ is $\{ \beta,-\gamma, \beta-\gamma,0\}$.
Moreover, the operator $\mathcal{A}_{L,L'}$ is diagonalizable.
\end{lemma}

\begin{proof} 
Let $P\in GL(2,\mathbb{C})$ such that $P^{-1} \tilde{A} P=\left(\begin{array}{cc}
\bar{\beta} & 0\\
0 & \bar{\gamma}\\
\end{array}\right)$. Notice that $\{\beta,\gamma\}\subset \{\bar{\beta},\bar{\gamma}\}$.
Denote by $E_{i,j}$ the elementary matrix which has $1$ as the coefficient situated on line $i$ and column $j$, and $0$ elsewhere. 

Case 1: $L\neq L'$. Here $\{\beta,\gamma\}=\{\bar{\beta},\bar{\gamma}\}$. Without loss of generality, assume $\beta= \bar{\beta}, \gamma = \bar{\gamma}$, that is to say, $L=L_1,L'=L_2$.
 Then $P_L=PE_{1,1}P^{-1}$ and $P_{L'}=PE_{2,2}P^{-1}$. Thus $PE_{1,1}P^{-1}$ is an eigenvector associated to $\bar{\beta}$ and $PE_{2,2}P^{-1}$ is  an eigenvector associated to $-\bar{\gamma}$.
The matrix $PE_{1,2}P^{-1}$ is an eigenvector associated to $\bar{\beta}-\bar{\gamma}$ and $PE_{2,1}P^{-1}$is in the kernel.

\bigskip
 Case 2: $L=L'=L_1$. Here $\beta=\gamma=\bar{\beta}$.
 Then $PE_{1,1}P^{-1}$ and $PE_{2,2}P^{-1}$ are in the kernel and $PE_{1,2}P^{-1},PE_{2,1}P^{-1}$ are eigenvectors associated to $\bar{\beta}$ and $-\bar{\beta}$ respectively.

\bigskip
Case 3: $L=L'=L_2$. This case is very similar to the previous one.
\end{proof}

\bigskip
Now assume $0$ is the only eigenvalue of $\tilde A$. If $\tilde A$ is the zero matrix, then $ad_{\tilde A}=0$. Otherwise $\tilde A$ is nilpotent and in this case one has the following lemma:

\begin{lemma}\label{nilpotent-tildeA}
Assume $\tilde A$ is nilpotent. Then the operator $ad_{\tilde A}$ 
has rank 2 and norm less  than 1, and is nilpotent of order 3.
\end{lemma}

\begin{proof} 
Let $P$ be such that 
$P^{-1}\tilde AP=\left(\begin{array}{cc}
0 & 1\\
0 & 0\\
\end{array}\right)$. Then $PE_{12}P^{-1}$ and $P(E_{11}+E_{22})P^{-1}$ are in the kernel. 
Moreover $ad_{\tilde A} PE_{11}P^{-1}= -PE_{12}P^{-1} $ and
$ad_{\tilde A} PE_{21}P^{-1}=P(E_{11}-E_{22})P^{-1}$, therefore $ad_{\tilde A}$ has norm less than 1. 

This also implies that $ad_{\tilde A}^2(PE_{1,1}P^{-1})=0$ and $ad_{\tilde A}^3(PE_{2,1}P^{-1})=0$. Finally $ad_{\tilde A}$ is nilpotent of order 3.
\end{proof}

\begin{proposition}\label{homologique}
Let $N \in \mathbb{N}, \kappa' \in ]0, \kappa], r \in ]0, r_0[$.
Let $\tilde{A} \in sl(2, \R)$ with $BR_\omega ^N(\kappa')$ spectrum. Let $\tilde{F} \in U_r(\T^d, sl(2,\R))$. Then there exists a solution $\tilde X \in$ $U_{r}( \T^d, sl(2,\R))$ of the equation
\begin{equation}\label{eq:homologique}
\forall \theta \in \T^d, \partial_\omega \tilde X(\theta) = [\tilde A, \tilde X( \theta)] + \tilde F^N - \hat {\tilde F}(0), \quad \hat {\tilde{X}}(0) = 0 
\end{equation}

The truncation of $\tilde X$ at order $N$ is unique.

Moreover,
\begin{enumerate}
\item \label{cas1}
if $\tilde A$ is diagonalizable with distinct eigenvalues,
let $\LL_{\tilde A}=\{L_1,L_2\}$ (the decomposition into eigenspaces of $\tilde A$) and $\Phi=P^\LL_{L_1}e^{2i\pi\langle m,\cdot\rangle }+P^\LL_{L_2}e^{-2i\pi\langle m,\cdot\rangle }$ for some $m\in\frac{1}{2}\mathbb{Z}^d$, $|m|\leq N$, such that for $i\in \{1,2\}$, $\Vert P^\LL_{L_i}\Vert \leq \frac{2C_0}{\kappa'^6}$, then

\[ \vert \Phi^{-1} \tilde X \Phi \vert _{r}  \leq 4C_0^2 \big( \frac{1}{\kappa'}\big)^{13} \Psi(N) \vert \Phi^{-1} \tilde F \Phi \vert_{r}\]

\item \label{cas3} if $\tilde A$ is nilpotent,

\[ \vert  \tilde X  \vert _{r}  \leq \frac{3}{\kappa^3}\Psi(N)^3 \vert  \tilde F  \vert_{r}\]

\item \label{cas4} if $ad_{\tilde{A}}=0$, then 

\[ \vert  \tilde X  \vert _{r}  \leq \frac{1}{\kappa} \Psi(N) \vert  \tilde F  \vert_{r}\]

\end{enumerate}

\end{proposition}

\begin{proof}
About existence, uniqueness and continuity of $\tilde X \in sl(2, \R)$ on $\T^d$, the proof is the same as \cite{1}, proposition 3.2. We now have to show the estimate which also follows from \cite{1} and we will adapt the proof to ultra-differentiable setting.

\bigskip
\textbf{Case \ref{cas1}:} $\tilde{A}$ has two $\kappa'$-separated eigenvalues. 
Let $\Phi =P_{L_1} e^{2 i \pi \langle m_1, . \rangle} + P_{L_2} e^{-2 i \pi \langle m_1, . \rangle}$
where $L_1$ and $L_2$ are the eigenspaces of $\tilde A$, and $|m_1|\leq N$. 
For all $L,L'\in \LL_{\tilde A}$, let the linear operator $\mathcal{A}_{L,L'} : gl(2,\R) \rightarrow gl(2, \R), M \mapsto \mathcal{A}_{L,L'} M := \tilde A P_LM - MP_{L'}\tilde A$. We decompose (\ref{eq:homologique}) into blocks, and we get for all $L, L' \in \LL_{\tilde A}$, 
\[
\partial_\omega (P_L \tilde X(\theta) P_{L'}) =\mathcal{A}_{L,L'}  P_L \tilde X (\theta) P_{L'} + P_L(\tilde F^N - \hat{\tilde F} (0))P_{L'}
\]

Then for all $m \in \frac{1}{2} \Z^d$, $0<\vert m \vert \leq N$,

\[
2i \pi \langle  m, \omega \rangle(P_L \hat{\tilde X}(m) P_{L'}) = \mathcal{A}_{L,L'}  (P_L \hat{\tilde X}(m) P_{L'} )+ P_L \hat{\tilde F}(m) P_{L'}
\]

Let 

\[ A_D:=(2i \pi \langle  m, \omega \rangle I - \mathcal{A}_{L,L'}) \]

By Lemma \ref{spectrum-ALL'}, $\sigma(\mathcal{A}_{L,L'}) = \{ \alpha - \alpha',\alpha,-\alpha',0 ; \alpha \in \sigma(\tilde A_{\vert L}), \alpha' \in \sigma(\tilde A_{\vert L'})  \}$, therefore $\sigma(\mathcal{A}_{L,L'}-2i \pi \langle  m, \omega \rangle I) = \{ \alpha - \alpha'-2i \pi \langle  m, \omega \rangle ,\alpha-2i \pi \langle  m, \omega \rangle,-\alpha'- 2i \pi \langle  m, \omega \rangle,-2i \pi \langle  m, \omega \rangle  ; \alpha \in \sigma(\tilde A_{\vert L}), \alpha' \in \sigma(\tilde A_{\vert L'})  \}$.
Moreover $\mathcal{A}_{L,L'}$ is diagonalizable, therefore $A_D$ as well, with non zero eigenvalues, and
$\Vert A_D^{-1}\Vert  =\max \{ \vert \beta\vert , \beta\in \sigma(A_D^{-1})\}
=\max \{\vert \gamma\vert ^{-1}, \gamma\in \sigma(A_D)\} $.

Since $\forall \alpha \in \sigma(\tilde A_{\vert L}), \alpha' \in \sigma(\tilde A_{\vert L'})$, $\vert \alpha - \alpha' - 2i\pi \langle m, \omega \rangle \vert \geq \frac{\kappa'}{\Psi( m )}  $ (for $m \in \Z^d$ if $L=L'$, $m \in\frac{1}{2}\Z^d$ if $L \neq L'$),
then  
\[  
\Vert (2i \pi \langle  m, \omega \rangle - \mathcal{A}_{L,L'})^{-1} \Vert \leq \big( \frac{\Psi(m )}{\kappa '}\big)
\]

Finally, for all $0<\vert m\vert \leq N$,
\[
\Vert P_L \hat{\tilde X}(m) P_{L'} \Vert = \Vert (2i \pi \langle  m  , \omega \rangle - \mathcal{A}_{L,L'})^{-1}P_L \hat{\tilde F}(m ) P_{L'} \Vert  \leq  
\big( \frac{\Psi(m )}{\kappa '}\big) \Vert P_L \hat{\tilde F}(m) P_{L'}  \Vert 
\]

Denoting by $m_L$ the vector appearing in $\Phi$ along the projection onto $L$, this estimate implies:

\[ \vert P_L \tilde X e^{2i\pi \langle m_L - m_{L'}, . \rangle}P_{L'}\vert_{r'} = \sum_{\vert m-m_L+m_{L'}\vert \leq N} \Vert P_L \hat{\tilde X}(m-m_L+m_{L'}) P_{L'} \Vert e^{2\pi\Lambda( m ) r'} \]
\[\leq 
\sum_{\vert m-m_L+m_{L'}\vert \leq N}  \Vert P_L \hat{\tilde F}(m-m_L+m_{L'} ) P_{L'} \Vert  e^{2\pi \Lambda( m ) r'} \frac{\Psi(\vert m-m_L+m_{L'} \vert)}{\kappa'} \]

\begin{equation}\label{estimate:homologique}\leq  
\frac{\Psi( N )}{\kappa'}  \vert P_L \tilde Fe^{2i\pi \langle m_L - m_{L'}, . \rangle} P_{L'} \vert_{r'}  
\end{equation}

We finally estimate $\vert \Phi^{-1} \tilde X \Phi \vert _{r'}$.

\[ \vert \Phi^{-1} \tilde X \Phi \vert _{ r'} = \vert \sum_{L, L' \in \LL} P_L \Phi^{-1} \tilde X \Phi  P_{L'} \vert _{ r'} = \vert \sum_{L, L' \in \LL} P_L \tilde X e^{2i\pi \langle m_L - m_{L'}, . \rangle}  P_{L'} \vert _{ r'} \]

therefore, from (\ref{estimate:homologique}), 

\[ \vert \Phi^{-1} \tilde X \Phi \vert _{r'} \leq  
\frac{\Psi( N ) }{\kappa'} \sum_{L, L' \in \LL}  \vert P_L \tilde F e^{2i\pi \langle m_L - m_{L'}, . \rangle} P_{L'} \vert_{r} 
=\frac{\Psi( N ) }{\kappa'} \sum_{L, L' \in \LL} \vert P_L\Phi^{-1}\tilde{F}\Phi P_{L'}\vert _{r}\]

therefore, since $\Vert P_L \Vert\leq \frac{2C_0}{\kappa^{'6}}$, we get the result

\[ \vert \Phi^{-1} \tilde X \Phi \vert _{r'}  \leq 4C_0^2 \big( \frac{1}{\kappa'}\big)^{13} \Psi( N ) \vert \Phi^{-1} \tilde F \Phi \vert_{ r}.\]

\bigskip
\textbf{Case \ref{cas3}:} $\tilde{A}$ is nilpotent. 
One has to estimate the inverse of the operator $ 2i\pi\langle m,\omega\rangle I-ad_{\tilde{A}}$. By Lemma \ref{nilpotent-tildeA}, 

\begin{equation}\begin{split}(2i\pi\langle m,\omega\rangle I-ad_{\tilde{A}})^{-1}& = (2i\pi\langle m,\omega\rangle)^{-1}
[I+(2i\pi\langle m,\omega\rangle)^{-1}ad_{\tilde{A}}\\
&+
(2i\pi\langle m,\omega\rangle)^{-2}ad_{\tilde{A}}^2 ]
\end{split}\end{equation}

\noindent Therefore 

$$
\Vert (2i\pi\langle m,\omega\rangle I-ad_{\tilde{A}})^{-1}\Vert 
\leq 3\vert 2i\pi\langle m,\omega\rangle\vert ^{-3}$$

\bigskip
Finally,  for all $0<\vert m\vert \leq N$,

\[
\Vert  \hat{\tilde X}(m)  \Vert = \Vert (2i \pi \langle  m  , \omega \rangle - \mathcal{A}_{L,L'})^{-1} \hat{\tilde F}(m ) \Vert 
\leq  
3\big( \frac{\Psi(m )}{\kappa }\big)^3 \Vert  \hat{\tilde F}(m)   \Vert 
\]

Thus,

\[ \vert \tilde X  \vert _{r'}  \leq \frac{3}{\kappa^3} \Psi( N )^3 \vert  \tilde F  \vert_{ r}\]

\textbf{Case \ref{cas4}:} The operator to invert is just $2i\pi \langle m,\omega\rangle I$, which makes the estimate much simpler.

\end{proof}

\section{Inductive lemma without renormalization}

Before stating the inductive lemma, we will need this next result which will allow us to iterate the inductive lemma without needing a new renormalization map at each step.

\begin{lemma}\label{341}
Let $\kappa' \in ]0,1[, \tilde F \in sl(2, \R), \tilde \varepsilon = \Vert \tilde F \Vert, \tilde N \in \N, \tilde A \in sl(2, \R)$ with $BR_\omega ^{\tilde N}(\kappa ')$ spectrum.

If \[\tilde \varepsilon  \leq  \big(\frac{\kappa '}{32(1 +\Vert \tilde A \Vert)}  \big)^{2}\frac{1}{\Psi(\tilde N)^2}, \]
then $\tilde A + \tilde F$ has $BR_\omega^{\tilde{N}}(\frac{3\kappa '}{4})$ spectrum.
\end{lemma}
\begin{proof}
 If $\tilde \alpha \in \sigma (\tilde A + \tilde F)$, there exists $\alpha \in \sigma(\tilde A)$ such that $\vert \alpha - \tilde \alpha \vert \leq 4(\Vert \tilde A \Vert +1)\tilde \varepsilon ^{\frac{1}{2 }}$ (see \cite{1}, lemma 4.1).
Since $\tilde A$ has $BR_\omega ^{\tilde N}(\kappa ')$ spectrum, for all $\alpha, \alpha' \in \sigma (\tilde A + \tilde F)$, for all $m \in \frac{1}{2}\Z^d$, $ 0<\vert m \vert \leq \tilde N$,

\[
\vert \alpha - \alpha' - 2i\pi \langle m, \omega \rangle \vert \geq \frac{\kappa '}{\Psi( m )} - 8(\Vert \tilde A \Vert +1)\tilde \varepsilon ^{\frac{1}{2}}
\]
We have to check that $8(\Vert \tilde A \Vert +1)\tilde \varepsilon ^{\frac{1}{2}} \leq \frac{\kappa '}{4\Psi( m)}$,  
which is satisfied by assumption.
\end{proof}

\begin{lemma}\label{NormeTroncation}
Let $N\geq 1$.
If $\LL=\{L_1,L_2\}$ is a decomposition of $\R^2$ into supplementary subspaces, and $\Phi$ is trivial with respect to $\LL$ of order $N$, then for all $0<r'<r$ and all $G\in U_r(\mathbb{T}^d,sl(2,\mathbb{R}))$,
\[ \vert \Phi^{-1}(G - G^{3N}) \Phi\vert_{r'} \leq e^{-2 \pi \Lambda(N)(r-r')} \vert \Phi^{-1}G \Phi \vert_{ r}  \] 
\end{lemma}

\begin{proof}
Write $\Phi =  P^\LL_{L_1} e^{2i\pi \langle m_1, \cdot \rangle}+ P^\LL_{L_2} e^{2i\pi \langle m_1, \cdot \rangle}$, $m_1\in\frac{1}{2}\mathbb{Z}^d$, then
\[ 
\vert \Phi^{-1}(G-G^{3N})\Phi \vert_{r'} = \vert \sum_{L,L'\in\LL}P^\LL_L(G-G^{3N})e^{2i\pi \langle m_L - m_{L'}, \cdot \rangle} P^\LL_{L'}\vert_{ r'} \]
\[ = \sum_{k\in\mathbb{Z}^d} \Vert  \sum_{L,L'} P^\LL_L  \widehat{(G-G^{3N})}(k-m_L +m_{L'}) P^\LL_{L'}\Vert e^{2\pi \Lambda(k) r'} \] 
\[=  \sum_{k\in\mathbb{Z}^d} \Vert \sum_{L,L'} P^\LL_L \widehat{(G-G^{3N})}(k-m_L +m_{L'}) P^\LL_{L'}\Vert e^{2\pi \Lambda(k)r} e^{2\pi \Lambda(k)(r'-r)}\]

\noindent Now if $\vert k\vert \leq N$, then for all $L,L'\in \LL$, $\vert k-m_L+m_{L'}\vert \leq 3N$, and $\sum_{L,L'\in \LL}P^\LL_L\widehat{(G-G^{3N})}(k-m_L+m_{L'})P^\LL_{L'}=0$, therefore

\[\vert \Phi^{-1}(G-G^{3N})\Phi \vert_{r'}\leq e^{2\pi \Lambda(N)(r'-r)} \sum_{\vert k\vert >N} \Vert \sum_{L,L'\in\LL} P^\LL_L \hat G(k-m_L +m_{L'}) P^\LL_{L'}\Vert e^{2\pi \Lambda(k)r}  \]
\[\leq e^{2\pi \Lambda(N)(r'-r)} \vert \sum_{L,L'\in\LL}   P^\LL_L G e^{2i\pi \langle m_L - m_{L'}, \cdot \rangle} P^\LL_{L'}\vert_{ r} = e^{2\pi \Lambda(N)(r'-r)} \vert \Phi^{-1}G \Phi \vert_{r} 
\] 
\end{proof}

We can now state the first inductive lemma, which does not require a renormalization map.

\begin{lemma}\label{344}
Let
\begin{itemize}
\item $\tilde \varepsilon >0, 
\tilde r>0, 0<\kappa' < 1, \tilde N \in \N^*,  \tilde r' < \tilde r $,
\item $\tilde F \in U_{\tilde r}(\T^d, sl(2,\R)), \tilde A \in sl(2,\R)$.
\end{itemize}

If
\begin{enumerate}
\item $\tilde A$ has $BR_\omega ^{\tilde N}(\kappa ')$ spectrum,
\item 
\[ \Vert \hat{\tilde F}(0) \Vert \leq \tilde \varepsilon  \leq \big(\frac{\kappa '}{32(1 + \Vert \tilde A \Vert)} \big)^{2}\frac{1}{\Psi(\tilde N)^2}\]
\end{enumerate}
then there exist
\begin{itemize}
\item $X \in U_{\tilde r'}(\T^d, sl(2, \R))$,
\item $A' \in sl(2, \R)$,
\end{itemize}
such that
\begin{enumerate}
\item $A'$ has $BR_\omega ^{\tilde N} (\frac{3\kappa '}{4})$ spectrum,
\item $\Vert A'- \tilde A\Vert \leq \tilde \varepsilon$, \newline
If $F' \in U_{ \tilde r '}(\T^d, sl(2, \R))$ is defined by

\begin{equation}\label{defF'}\forall \theta \in  \T^d, \partial_\omega e^{X(\theta)} = (\tilde A + \tilde F(\theta))e^{X(\theta)} - e^{X(\theta)}(A' + F'(\theta)) ,
\end{equation}

\noindent then we have the following estimates :

\bigskip
If $\tilde A$ has two different eigenvalues, if $\Phi$ is of the form
$\Phi = P_{L_1}e^{2i\pi\langle m,\cdot\rangle}+P_{L_2}e^{-2i\pi\langle m,\cdot\rangle}$ where $L_1,L_2$ are the two eigenspaces of $\tilde A$, $|m|\leq \tilde N$ and $\Vert P_{L_i} \Vert \leq \frac{2C_0}{\kappa^{'6}}$,
\item \label{estim3}
\[ \vert \Phi^{-1}X\Phi\vert_{ \tilde r'}\leq 4C_0^2 \big( \frac{1}{\kappa'}\big)^{13} \Psi(3 \tilde N) \vert \Phi^{-1} \tilde F \Phi \vert_{ \tilde r},\]
\item \label{estim4}
\[\vert \Phi^{-1} F' \Phi \vert_{ \tilde r'}  \leq 4C_0^2 e^{\vert \Phi^{-1}X\Phi \vert_{ \tilde r'}} \big( \frac{1}{\kappa'}\big)^{13} \vert \Phi^{-1}\tilde F \Phi \vert_{ \tilde r} \big[ e^{-2\pi \Lambda(\tilde N)(\tilde r-\tilde r')} \qquad \]
\[  \qquad +  \vert \Phi^{-1}\tilde F \Phi \vert_{ \tilde r} \Psi(3\tilde N)(2e+e^{\vert \Phi^{-1} X \Phi \vert_{ \tilde r'}})  \big].\]

\bigskip
If $\tilde A$ is nilpotent:
\item \label{estim5}
\[ \vert X\vert_{ \tilde r'}\leq \frac{3}{\kappa^3}\Psi( 3\tilde N)^3 \vert  \tilde F  \vert_{ \tilde r},\]

\item \label{estim7}
\[\vert  F'  \vert_{ \tilde r'}  \leq \frac{3}{\kappa^3}e^{\vert X \vert_{ \tilde r'}} \vert \tilde F  \vert_{ \tilde r} \big[ e^{-2\pi \Lambda(\tilde N)(\tilde r-\tilde r')} \qquad \]
\[  \qquad +  \vert \tilde F  \vert_{ \tilde r} \Psi(3\tilde N)^3(2e+e^{\vert  X  \vert_{ \tilde r'}})  \big].\]

\bigskip
If $ad_{\tilde A}=0$:
\item \label{estim8} \[ \vert X\vert_{ \tilde r'}\leq \frac{1}{\kappa} \Psi(3\tilde N) \vert  \tilde F  \vert_{ \tilde r}, \]
\item \label{estim9}
\[\vert  F'  \vert_{ \tilde r'}  \leq \frac{1}{\kappa}e^{\vert X \vert_{ \tilde r'}}  \vert \tilde F  \vert_{ \tilde r} \big[ e^{-2\pi \Lambda(\tilde N)(\tilde r-\tilde r')} \qquad \]
\[  \qquad +  \vert \tilde F  \vert_{ \tilde r} \Psi(3\tilde N)(2e+e^{\vert  X  \vert_{\Lambda, \tilde r'}})  \big].\]

In any case, there is the estimate 
\item \label{estim6} \[\vert \partial_\omega X\vert _{\tilde r'}\leq 2\Vert \tilde A\Vert \ \vert X\vert _{\tilde r'}+\vert \tilde F\vert _{\tilde r'}.
\]

\end{enumerate}
\end{lemma}

\begin{proof}
By assumption, $\tilde A$ has $BR_\omega^{\tilde N}(\kappa ')$ spectrum, so we apply Proposition \ref{homologique} with $N=3\tilde N$. Let $X \in U_{ \tilde r}(\T^d, sl(2, \R))$ a solution of
\[ \forall \theta \in 2\T^d, \partial_\omega X(\theta) = [\tilde A, X(\theta)] + \tilde F^{3\tilde N}(\theta) - \hat{\tilde F}(0)\]
satisfying the conclusion of Proposition \ref{homologique}. This obviously implies the property \ref{estim6}.

Let $A' := \tilde A + \hat{\tilde F}(0)$. We have $A' \in sl(2, \R)$ and $\Vert \tilde A - A' \Vert = \Vert \hat{\tilde F}(0) \Vert$, and then property 2.
With assumption (2) we can apply lemma \ref{341} to deduce that $A'$ has $BR_\omega^{\tilde N}(\frac{3\kappa'}{4})$ spectrum, and then property 1.

If $F'$ is defined in equation (\ref{defF'}), 

\begin{equation}\label{expression-F'} F' = e^{-X}(\tilde F - \tilde F ^{3\tilde N}) + e^{-X}\tilde F (e^X-Id) + (e^{-X}-Id)\hat{\tilde F}(0) - e^{-X} \sum_{k \geq 2}\frac{1}{k!}\sum_{l=0}^{k-1}X^l(\tilde F ^{3\tilde N}-\hat{\tilde F}(0))X^{k-1-l} \end{equation}

$\bullet$ Case 1: $\tilde A$ has two different eigenvalues :
Let $\Phi$ be as required, then 
\[ 
\vert \Phi^{-1} F' \Phi \vert_{ \tilde r'} \leq  e^{\vert \Phi ^{-1}X \Phi  \vert_{\tilde r'}} [\vert \Phi^{-1}(\tilde F - \tilde F^{3\tilde N})\Phi  \vert_{\tilde r'} + \vert \Phi^{-1}\tilde F \Phi \vert_{ \tilde r}\vert \Phi^{-1}X \Phi  \vert_{ \tilde r'}(2e+e^{\vert \Phi^{-1}X\Phi \vert_{\tilde r'}})]
\]

From proposition \ref{homologique}, estimate \ref{cas1},

\[\vert \Phi^{-1}  X \Phi \vert_{ \tilde r'}   \leq \vert \Phi^{-1}  X \Phi \vert_{ \tilde r}   \leq  4C_0^2 \big( \frac{1}{\kappa'}\big)^{13} \Psi(3\tilde N)  \vert \Phi^{-1} \tilde F \Phi \vert_{ \tilde r}\]
whence \eqref{estim3}; and from lemma \ref{NormeTroncation}, since $\tilde r' < \tilde r$,
\[ \vert \Phi^{-1}(\tilde F - \tilde F ^{3\tilde N}) \Phi \vert_{ \tilde r'} \leq e^{-2 \pi \Lambda(\tilde N)(\tilde r- \tilde r')} \vert\Phi^{-1} \tilde F \Phi\vert_{ \tilde r}   \]
which finally gives

\[
\vert \Phi^{-1} F' \Phi \vert_{ \tilde r'} \leq  e^{\vert \Phi ^{-1}X \Phi  \vert_{\tilde r'}} [e^{-2 \pi \Lambda(\tilde N)(\tilde r- \tilde r')} \vert\Phi^{-1} \tilde F \Phi\vert_{ \tilde r} \qquad \qquad\qquad\qquad \qquad \qquad \] \[ \qquad\qquad \qquad \qquad \qquad + \vert \Phi^{-1}\tilde F \Phi \vert_{ \tilde r}4C_0^2 \big( \frac{1}{\kappa'}\big)^{13} \Psi(3 \tilde N)  \vert \Phi^{-1} \tilde F \Phi \vert_{ \tilde r}(2e+e^{\vert \Phi^{-1}X\Phi \vert_{ \tilde r'}})]
\]
\[ \leq 4C_0^2 e^{\vert \Phi^{-1}X\Phi \vert_{ \tilde r'}} \big( \frac{1}{\kappa'}\big)^{13} \vert \Phi^{-1}\tilde F \Phi \vert_{\tilde r} \big[ e^{-2\pi \Lambda(\tilde N)(\tilde r-\tilde r')} +  \vert \Phi^{-1}\tilde F \Phi \vert_{ \tilde r} \Psi( 3\tilde N)(2e+e^{\vert \Phi^{-1} X \Phi \vert_{ \tilde r'}})  \big]
\]  
hence \ref{estim4} holds.

\bigskip
$\bullet$ Case 2: $\tilde A$ is nilpotent : \eqref{expression-F'} implies

\[ 
\vert F'  \vert_{ \tilde r'} \leq  e^{\vert X   \vert_{\tilde r'}} [\vert \tilde F - \tilde F^{3\tilde N}  \vert_{\tilde r'} + \vert \tilde F  \vert_{ \tilde r}\vert X   \vert_{ \tilde r'}(2e+e^{\vert X \vert_{\tilde r'}})]
\]

From proposition \ref{homologique}, estimate \ref{cas3},

\[\vert X  \vert_{ \tilde r'}   \leq \vert   X  \vert_{ \tilde r}   \leq  \frac{3}{\kappa^3} \Psi(3\tilde N)^3  \vert  \tilde F \vert_{ \tilde r}\]
which is estimate \ref{estim5}.
Moreover, from Lemma \ref{NormeTroncation},

\[|\tilde F - \tilde F^{3\tilde N}|_{\tilde r'}\leq e^{-2\pi \Lambda (\tilde N)(\tilde r-\tilde r')}|\tilde F|_{\tilde r}\]

Therefore, similarly to the previous case, we get
\[\vert  F'  \vert_{ \tilde r'}  \leq \frac{3}{\kappa^3} e^{\vert X \vert_{ \tilde r'}} \vert \tilde F  \vert_{ \tilde r} \big[ e^{-2\pi \Lambda(\tilde N)(\tilde r-\tilde r')} \qquad \]
\[  \qquad +  \vert \tilde F  \vert_{ \tilde r} \Psi(3\tilde N)^3(2e+e^{\vert  X  \vert_{ \tilde r'}})  \big]\]
which is estimate \ref{estim7}. 

\bigskip
$\bullet$ Case 3: $ad_{\tilde A}=0$  :
From proposition \ref{homologique}, estimate \ref{cas4}
\[ \vert  \tilde X  \vert _{\tilde r}  \leq \frac{1}{\kappa} \Psi(3\tilde N) \vert  \tilde F  \vert_{\tilde r} \]
which is estimate  \eqref{estim8}, and similarly to the two previous cases, we get the estimate \eqref{estim9}:
\[\vert  F'  \vert_{ \tilde r'}  \leq 
\frac{1}{\kappa}
e^{\vert X \vert_{ \tilde r'}}  \vert \tilde F  \vert_{ \tilde r} \big[ e^{-2\pi \Lambda(\tilde N)(\tilde r-\tilde r')} \qquad \]
\[  \qquad +  \vert \tilde F  \vert_{ \tilde r} \Psi(3\tilde N)(2e+e^{\vert  X  \vert_{ \tilde r'}})  \big].\]
\end{proof}

\section{Inductive lemma with renormalization}

The following Lemma is used to define the smallness assumption on $\epsilon_0$ mentioned in section \ref{smallness-assumption}. This smallness assumption shall be sufficient for Lemmas \ref{351} and \ref{352}.

\begin{lemma}\label{smallness-epsilon}

Let $l=56$.
There exists $\varepsilon_0 >0$ depending on $C_0$, $C'$ $\kappa$, $b_0$ and $D_5$, such that, for all $\varepsilon \in ]0, \varepsilon_0]$,
the following inequalities hold for all $2 \leq j \leq l$ : \\
\textbf{In lemma \ref{351}} 
\begin{equation} \frac{1}{2}\kappa\varepsilon^{\frac{1}{1728}}+\varepsilon^{\frac{845}{864}} \leq \frac{3}{4}\kappa \varepsilon^{\frac{1}{1728}} \label{eq:cond1.1}\end{equation}

\begin{equation} \frac{4C_0^2}{\kappa^{13}}\varepsilon^{-13\zeta}\varepsilon^{-3\zeta}\varepsilon^{1-2\zeta} \leq \varepsilon^{\frac{7}{8}} \label{eq:cond1.2}\end{equation}

\begin{equation} 8C_0^2\varepsilon^{1-2\zeta-\frac{1}{96}}(\varepsilon^{100\delta}+3\varepsilon^{1-6\zeta})
\leq \varepsilon^{\frac{3}{2}-4\zeta-\frac{1}{96}} \label{eq:cond1.3}
\end{equation}

\textbf{In lemma \ref{352}}

\begin{equation}\label{eq:cond2.0.0} \varepsilon^{1-576\zeta} \leq  (2C_0)^{-96}\big(\frac{\kappa}{32(\varepsilon^{-\frac{\zeta}{2}} +1)}\big)^{576}
\end{equation}

\begin{equation} \label{eq:cond2.0}
\varepsilon^{\frac{5}{4}-\frac{1}{48}} \leq \big(\frac{\frac{3}{4}\frac{\kappa}{C_0}\varepsilon^{\zeta}}{32(1+(1+\pi)\varepsilon^{-\frac{\zeta}{2}}+\varepsilon^{\frac{23}{24}})} \big)^2\varepsilon^{2\zeta}
\end{equation}
\begin{equation}\varepsilon^{(\frac{5}{4})^j-\frac{1}{48}} \leq 
\big(\frac{(\frac{3}{4})^j\frac{\kappa }{C_0}\varepsilon^{\zeta}}{32(1+ \varepsilon^{\frac{23}{24}} + (1+\pi)\varepsilon^{-\frac{\zeta}{2}} + \sum_{i=1}^{j-1} \varepsilon^{(\frac{5}{4})^i-\frac{1}{96}})}\big)^2\varepsilon^{2\zeta} \label{eq:cond2.1}\end{equation}

\begin{equation}
256C_0^2 \varepsilon^{-14\zeta}\big( \frac{1}{(\frac{3}{4})^{j-1}\frac{\kappa}{C_0} }\big)^{13} \varepsilon^{(\frac{5}{4})^{j-1}} ( \varepsilon^{\frac{50\delta}{l}} +  \varepsilon^{(\frac{5}{4})^{j-1}}  ) \leq \varepsilon^{(\frac{5}{4})^j}
\label{eq:cond2.2} 
\end{equation}

\begin{equation} \varepsilon^{\frac{23}{24}} + \pi \varepsilon^{-\frac{\zeta}{2}} + \sum_{i=1}^l  \varepsilon^{(\frac{5}{4})^i-\frac{1}{48}} \leq \varepsilon^{-\zeta}  \label{eq:cond2.3}\end{equation}

\begin{equation} \frac{1}{2}\kappa \varepsilon^{\zeta} + 2\varepsilon^{\frac{5}{4}-\frac{1}{48}} \leq \kappa \varepsilon^{\zeta} \label{eq:cond2.4}\end{equation}

\begin{equation}
 \varepsilon^{-\frac{\zeta}{2}} + \varepsilon^{\frac{23}{24}} + \pi \varepsilon^{-\zeta} \leq \varepsilon^{-2\zeta}   \label{eq:cond2.5.2}
\end{equation}
\begin{equation} 
4 \varepsilon^{-2\zeta + \frac{59}{48}} + 2 \varepsilon^{\frac{5}{4}-\frac{1}{48}} \leq \varepsilon
\label{eq:cond2.6} \end{equation}

\begin{equation}2 \varepsilon^{\frac{1}{2}} + 2 \varepsilon^{\frac{7}{8}} \leq \varepsilon^{\frac{1}{4}}\label{eq:cond2.7} \end{equation}

\end{lemma}

\begin{proof}
\textbf{Equations in lemma \ref{351}} \\

Equation \eqref{eq:cond1.1} holds for \[\varepsilon \leq (\frac{1}{4}\kappa)^{\frac{1728}{2123}}.\]

Equation \eqref{eq:cond1.2} holds for
\[\varepsilon \leq \big(\frac{4C_0^2}{\kappa^{13}} \big)^{-\frac{96}{11}}.\]

Equation \eqref{eq:cond1.3} holds if
\[8C_0^2\varepsilon^{\frac{427}{432}}(\varepsilon^{100\delta}+3\varepsilon^{\frac{287}{288}})
\leq \varepsilon^{\frac{1285}{864}} \Leftrightarrow 8C_0^2(\varepsilon^{100\delta}+3\varepsilon^{\frac{287}{288}}) \leq \varepsilon^{\frac{431}{864}}\]
therefore, if we have
\[ \left\{\begin{array}{c}
8C_0^2\varepsilon^{100\delta}\leq \frac{1}{2}\varepsilon{\frac{431}{864}} \\
24C_0^2 \varepsilon^{\frac{287}{288}}\leq \frac{1}{2}\varepsilon{\frac{431}{864}}
\end{array}\right. 
\]
which is satisfied if \[\varepsilon \leq (48C_0^2)^{-\frac{432}{215}} \] then inequality \eqref{eq:cond1.3} holds.

\textbf{Equations in lemma \ref{352}} \\

Equation \eqref{eq:cond2.0.0} holds if
\[\varepsilon^{\frac{1}{1152}}+\varepsilon^{\frac{1}{864}} \leq (2C_0)^{-\frac{1}{6}}\frac{\kappa}{32} \]
which is satisfied if
\[ \left\{\begin{array}{c}
\varepsilon^{\frac{1}{1152}} \leq \frac{\kappa}{64}(2C_0)^{-\frac{1}{6}}\\
\varepsilon^{\frac{1}{864}} \leq \frac{\kappa}{64}(2C_0)^{-\frac{1}{6}}
\end{array}\right. 
 \Leftrightarrow 
  \left\{\begin{array}{c}
\varepsilon \leq (\frac{\kappa}{64})^{1152}(2C_0)^{-192}\\
\varepsilon \leq (\frac{\kappa}{64})^{864}(2C_0)^{-144}
\end{array}\right. 
 \]
 
Equation \eqref{eq:cond2.0} is satisfied if
\[ \varepsilon^{\frac{5}{4}-\frac{1}{48}-4\zeta}(1+ \varepsilon^{\frac{23}{24}} + (1+\pi) \varepsilon^{-\frac{\zeta}{2}})^2 \leq 
(\frac{3\kappa}{128C_0})^2. \]
For $\varepsilon \leq 1$, we have 
\[ 1+ \varepsilon^{\frac{23}{24}} + (1+\pi)\varepsilon^{-\frac{\zeta}{2}} \leq 4\pi \varepsilon^{-\frac{\zeta}{2}}\]
then we need
\[16\pi^2 \varepsilon^{\frac{5}{4}-\frac{1}{48}-5\zeta} \leq (\frac{3\kappa}{128C_0})^2 \]
which is satisfied if
\[ \varepsilon \leq (\frac{3\kappa}{512\pi C_0})^{\frac{2119}{1728}}\]
Equation \eqref{eq:cond2.1} is satisfied if
\[ \varepsilon^{(\frac{5}{4})^{j-1}-\frac{1}{96}-4\zeta}(1+ \varepsilon^{\frac{23}{24}} + (1+\pi) \varepsilon^{-\frac{\zeta}{2}}+ 2 \varepsilon^{\frac{5}{4}-\frac{1}{48}})^2 \leq 
(\frac{3}{4})^{2j}(\frac{\kappa}{32C_0})^2. \]
If $\varepsilon \leq 1$, then \[1+\varepsilon^{\frac{23}{24}} + (1+\pi) \varepsilon^{-\frac{\zeta}{2}}+ 2 \varepsilon^{\frac{5}{4}-\frac{1}{96}} \leq 4\pi \varepsilon^{-\frac{\zeta}{2}}\]
then it's enough to have
\[16\pi^2\varepsilon^{\frac{5}{4}-\frac{1}{48}-4\zeta-\zeta} \leq  (\frac{3}{4})^{2\cdot 104}(\frac{\kappa}{32C_0})^2 \]
which is satisfied if
\[\varepsilon \leq \big((\frac{3}{4})^{208}(\frac{\kappa}{128\pi C_0})^2) \big)^{\frac{3456}{2119}} \]

Equation \eqref{eq:cond2.2} holds if

\[ 256C_0^2(\frac{4}{3})^{13(j-1)}(\frac{C_0}{\kappa})^{13}( \varepsilon^{\frac{50\delta}{l}} +  \varepsilon^{(\frac{5}{4})^{j-1}}  ) \leq \varepsilon^{14\zeta+\frac{1}{4}(\frac{5}{4})^{j-1}}.\]
We will first show that, for all $j \in \llbracket 2, l \rrbracket$, and for $\varepsilon$ small enough,
\[\varepsilon^{\frac{50\delta}{l}} +  \varepsilon^{(\frac{5}{4})^{j-1}}\leq \varepsilon^{\frac{1}{3}(\frac{5}{4})^{j-1}}. \]
Since $l=56$, this condition is satisfied if for all $j\in \llbracket 2, l \rrbracket$ if
\[ \left\{\begin{array}{c}
2\leq \varepsilon^{\frac{1}{3}(\frac{5}{4})^{j-1}-\frac{50\delta}{l}} \\
2 \leq \varepsilon^{-\frac{1}{3}(\frac{5}{4})^{j-1}} 
 \end{array}\right. \]
 
\noindent which holds if 

\[ \left\{\begin{array}{c}
\varepsilon \leq 2^{\frac{1}{\frac{1}{2}(\frac{5}{4})^{55}-\frac{50\delta}{56}}} \\
\varepsilon \leq 2^{-\frac{12}{25}}
 \end{array}\right.\]
then equation  \eqref{eq:cond2.2}  is satisfied if
\[ 256C_0^2(\frac{4}{3})^{13(j-1)}(\frac{C_0}{\kappa})^{13} \leq \varepsilon^{14\zeta - \frac{1}{12}(\frac{5}{4})^{j-1}}\Leftrightarrow \varepsilon \leq \big(256C_0^{15}(\frac{4}{3})^{13(j-1)}(\frac{1}{\kappa})^{13}\big)^{\frac{1}{14\zeta - \frac{1}{12}(\frac{5}{4})^{j-1}}}.\]
Now, as $C_0 \geq 1$ and $0<\kappa<1$, since $\varepsilon \leq 1$,
\[ \big(\frac{24C_0^{15}}{\kappa^{13}}\big)^{\frac{1}{5\zeta - \frac{1}{12}(\frac{5}{4})^{j-1}}} \geq \big(\frac{24C_0^{15}}{\kappa^{13}}\big)^{\frac{1}{14\zeta - \frac{1}{12}(\frac{5}{4})}} =  \big(\frac{24C_0^{15}}{\kappa^{13}}\big)^{-\frac{864}{83}} \]
and
\[(\frac{4}{3})^{\frac{13(j-1)}{14\zeta - \frac{1}{12}(\frac{5}{4})^{j-1}}} \geq (\frac{4}{3})^{\frac{13\cdot 4}{14\zeta - \frac{1}{12}(\frac{5}{4})^{4}}} = (\frac{4}{3})^{-\frac{1437696}{5401}}\]
Finally, equation \eqref{eq:cond2.2} is satisfied with

\[\varepsilon\leq \big(\frac{24C_0^{15}}{\kappa^{13}}\big)^{-\frac{1728}{178}}(\frac{4}{3})^{-\frac{1437696}{5545}}\]

Equation \eqref{eq:cond2.3} is satisfied if
\[ \varepsilon^{\frac{23}{25}} + \pi \varepsilon^{-\frac{\zeta}{2}} + 2\varepsilon^{\frac{5}{4}-\frac{1}{48}} \leq \varepsilon^{-\zeta}.\]
So if we have
\[ \left\{\begin{array}{c}
\varepsilon^{\frac{23}{24}} \leq \frac{1}{10}\varepsilon^{-\zeta} \\
\pi \varepsilon^{-\frac{\zeta}{2}} \leq \frac{4}{5}\varepsilon^{-\zeta} \\
2\varepsilon^{\frac{59}{48}} \leq \frac{1}{10}\varepsilon^{-\zeta}
\end{array}\right. \Leftrightarrow
\left\{\begin{array}{c}
\varepsilon \leq (\frac{1}{10})^{\frac{1728}{1628}} \\
\varepsilon \leq (\frac{4\pi}{5})^{3456} \\
\varepsilon \leq (\frac{1}{20})^{\frac{1728}{1837}}
\end{array}\right.
\]
then equation \eqref{eq:cond2.3} holds.

Equation \eqref{eq:cond2.4} holds for
\[\varepsilon \leq \big( \frac{\kappa}{4} \big)^{\frac{1728}{2141}}. \]

Equation \eqref{eq:cond2.5.2} holds if

\[
\left\{\begin{array}{c}
\varepsilon^{-\frac{\zeta}{2}} \leq \frac{1}{3}\varepsilon^{-2\zeta} \\
\varepsilon^{\frac{23}{24}} \leq \frac{1}{3}\varepsilon^{-2\zeta} \\
\pi\varepsilon^{-\zeta}  \leq \frac{1}{3}\varepsilon^{-2\zeta}
\end{array}\right.
\Leftrightarrow 
\left\{\begin{array}{c}
\varepsilon \leq (\frac{1}{3})^{1152} \\
\varepsilon \leq (\frac{1}{3})^{\frac{864}{829}}\\
\varepsilon \leq (\frac{1}{3\pi})^{1728}
\end{array}\right.
\]

Equation \eqref{eq:cond2.6} holds if

\[
\left\{\begin{array}{c}
4\varepsilon^{\frac{1061}{864}} \leq \frac{1}{2}\varepsilon \\
2\varepsilon^{\frac{59}{48}}  \leq \frac{1}{2}\varepsilon 
\end{array}\right.
\Leftrightarrow 
\left\{\begin{array}{c}
\varepsilon \leq (\frac{1}{8})^{\frac{864}{197}} \\
\varepsilon \leq (\frac{1}{4})^{\frac{48}{11}}

\end{array}\right.
\]
and then equation \eqref{eq:cond2.6} holds.

Since, for $\varepsilon \leq 1$ we have $\varepsilon^{\frac{7}{8}}\leq \varepsilon^{\frac{1}{2}}$, equation \eqref{eq:cond2.7} holds if
\[ 4\varepsilon^{\frac{1}{2}} \leq \varepsilon^{\frac{1}{4}} \]
that's it to say, if
\[ \varepsilon \leq \frac{1}{256}.\]

Now define $\varepsilon_0$ in order to satisfy conditions \eqref{eq:cond1.1} to \eqref{eq:cond2.7}. 

\end{proof}

\begin{lemma}[Inductive lemma with renormalization]\label{351}
Let

\begin{itemize}
\item $A \in sl(2, \R)$,
\item $r>0$,
\item $\bar A, \bar F \in U_{r}(\T^d, sl(2, \R)), \psi \in U_{r}(2\T^d, SL(2,\R))$,
\item $\vert \bar F \vert_r = \varepsilon$,
\item
\[ N= \Lambda^{-1}\big(\frac{50     \vert \log \varepsilon \vert}{\pi r }\big) \]
\item
\[  R = \frac{1}{3N}\Psi^{-1}(\varepsilon^{-\zeta})  \]
\item
\[ r' = r - \frac{50 \delta \vert \log \varepsilon \vert}{\pi \Lambda(R N)}\]
\end{itemize}

Assume $r'>0$.
Let $\displaystyle \kappa '' = \frac{\kappa }{\Psi(3RN)} = \kappa \varepsilon^{\zeta}$. Suppose that $\varepsilon\leq \varepsilon_0$ which was defined in Lemma \ref{351} and

\begin{enumerate}
\item \label{351-epsilon-small}
\[\varepsilon \leq  (2C_0)^{-96} \big( \frac{ \kappa''}{32(\Vert A \Vert +1)} \big)^{576}\] 
\item $\bar A$ is reducible to $A$ by $\psi$,
\item \label{normeA} $\Vert A \Vert \leq \varepsilon^{-\frac{\zeta}{2}}$,
\item \label{period-psi}for all $G\in \mathcal{C}^0(\mathbb{T}^d, sl(2,\mathbb{R}))$, $\psi^{-1}G\psi \in \mathcal{C}^0(\mathbb{T}^d, sl(2,\mathbb{R}))$,
\item \label{boundedPsi} $\vert \psi ^{\pm 1}\vert_{ r} \leq \varepsilon^{-\zeta }$,
\end{enumerate}
then there exist
\begin{itemize}
\item $Z' \in U_{ r'}(\T^d, SL(2, \R))$,
\item $\bar A', \bar F' \in U_{ r'}(\T^d, sl(2,\R))$,
\item $\psi ' \in U_{ r}(2\T^d, SL(2,\R))$,
\item $A'\in sl(2, \R)$
\end{itemize}
satisfying the following properties :

\begin{enumerate}
\item \label{A'red} $\bar A'$ is reducible by $\psi '$ to $A'$,
\item \label{period-psi'} for all $G\in \mathcal{C}^0(\mathbb{T}^d, sl(2,\mathbb{R}))$, $\psi'^{-1}G\psi' \in \mathcal{C}^0(\mathbb{T}^d, sl(2,\mathbb{R}))$,
\item \label{spectrum-A'} $A'$ has $BR_\omega^{RN}(\frac{3}{4C_0}\kappa '')$ spectrum, where $C_0$ was defined in Lemma \ref{renormalization},
\item \label{conjug} \[ \partial _\omega Z'  = (\bar A + \bar F) Z' - Z' (\bar A ' + \bar F')\]
\item \label{norme-A'} $\Vert A' \Vert \leq \Vert A \Vert + \varepsilon^{\frac{23}{24}} + \pi N$,
\item \label{estim-Z'} \[\vert Z'^{\pm 1} - Id \vert _{ r'} \leq \varepsilon^{\frac{8}{9}}\]
\item  \label{estim-renorm} for all $s' >0$,
\[ \vert \psi'^{-1}\psi \vert_{ s'} \leq  2C_0 \big( \frac{1}{\kappa ''}\big)^{6}e^{2\pi \Lambda(\frac{N}{2})s'} \]\[ \vert \psi^{-1}\psi ' \vert_{s'} \leq 2 C_0 \big( \frac{1}{\kappa ''}\big)^{6}e^{2\pi \Lambda(\frac{N}{2})s'},\]
\item \label{estim-psi'}$\vert \psi'^{\pm 1} \vert_{ r} \leq \varepsilon^{-\zeta-\frac{1}{96}}e^{2\pi  \Lambda(\frac{N}{2})r}$,
\item \label{estim-F'} $\vert \psi^{-1}\bar F' \psi \vert_{r'} \leq \varepsilon^{\frac{5}{4}}$.
\item \label{norme2-A'} If moreover the spectrum of $A$ is not $BR_\omega^{RN}(\kappa'')$, then $\Vert A' \Vert \leq \frac{3}{4}\kappa''$, 
\item \label{estim-Z'-deriv} If the spectrum of $A$ is $BR_\omega^{RN}(\kappa'')$, we have $\Phi \equiv I$, then $\psi' = \psi$ and $\tilde A=A$, and  then 
 \[\vert \psi^{-1}Z'^{\pm 1}\psi \vert_{r'} \leq e^{\varepsilon^{\frac{7}{8}}}, \]
\[\vert \partial_\omega(\psi^{-1}Z'^{\pm 1}\psi) \vert_{r'} \leq \varepsilon^{\frac{1}{2}},\]

\end{enumerate}\end{lemma}
\begin{proof} \textbf{Algebraic aspects} \\

If $A$ has a double eigenvalue or $\kappa''$-close eigenvalues, let $\Phi$ be defined on $2\mathbb{T}^d$ as constantly equal to $I$ and let $\tilde A=A$.
Otherwise, let $\Phi$ a renormalization of $A$ of order $R, N$ given by lemma \ref{renormalization}. Let $\tilde A \in sl(2,\R)$ such that
\[ \forall \theta \in 2\T^d, \partial_\omega \Phi(\theta) = A \Phi(\theta)-\Phi(\theta)\tilde A \]

\noindent so $||A-\tilde A||\leq \pi N$ and $\tilde A$ has $BR^{RN}_\omega(\kappa'')$ spectrum.
Notice that in this case, $\tilde A$ is not nilpotent. 
Let $\psi' = \psi \Phi$, and
\[ \tilde F:= \psi'^{-1} \bar F \psi ' \]

Moreover, $\Phi$ is trivial with respect to $\LL_{A}$ :

\begin{equation}\label{trivialPhi-step} \Phi = P_{L_1}^{\LL_{A}} e^{2i\pi \langle m, \cdot \rangle} + P_{L_2}^{\LL_{A}} e^{-2i\pi \langle m, \cdot \rangle} 
\end{equation}

\noindent
with $|m|\leq N$ and $||P_{L_i}||\leq \frac{C_0}{\kappa^{''6}}$.
Since $\Phi$ is trivial with respect to $\LL_{A}$, for all $s' \geq 0$, Lemma \ref{renormalization} implies
\begin{equation}\label{estimPhi-prouvee} \vert \Phi ^{\pm1} \vert_{s'} \leq 2C_0\Big( \frac{1 }{\kappa ''} \Big)^{6}e^{2\pi \Lambda(\frac{N}{2})s'} 
\end{equation}
which gives property \ref{estim-renorm}. 

Let $\psi'=\psi\Phi$. 
Let $G\in \mathcal{C}^0(\mathbb{T}^d,sl(2,\mathbb{R}))$, then by triviality of $\Phi$, $\Phi^{-1}G\Phi\in \mathcal{C}^0(\mathbb{T}^d,sl(2,\mathbb{R}))$, and by the assumption \ref{period-psi}, $\psi^{'-1}G\psi'\in \mathcal{C}^0(\mathbb{T}^d,sl(2,\mathbb{R}))$. Therefore the property \ref{period-psi'} on $\psi'$ holds.

\bigskip

\noindent
Moreover,
\[ \Vert \hat {\tilde F}(0) \Vert \leq \vert \tilde F \vert_{0} \leq \vert \Phi \vert_{ 0}  \vert \Phi ^{-1} \vert_{ 0}  \vert \psi \vert_{ 0}  \vert \psi^{-1} \vert_{ 0}   \vert \bar F \vert_{ 0} \]
Therefore by \eqref{estimPhi-prouvee} and by assumption \ref{boundedPsi},

\[ \Vert \hat{\tilde F}(0) \Vert \leq \varepsilon^{1-2\zeta}(2 C_0)^2 \big( \frac{1}{\kappa ''}\big)^{12}.\]
Since $\varepsilon \leq (2C_0)^{-96} \big( \frac{ \kappa''}{32(\Vert A \Vert +1)} \big)^{576}
\leq  (2C_0)^{-96} \kappa''^{576 }$, we get
\[ \Vert \hat{ \tilde F}(0) \Vert \leq \varepsilon^{1-2\zeta-\frac{1}{48}}. \]
Since $\tilde A$ has a $BR_\omega ^{R N}(\kappa '')$ spectrum, we want to apply lemma \ref{344} with 
\[ \tilde \varepsilon = \varepsilon^{1-2\zeta-\frac{1}{48}}, \tilde r = r, \tilde r' = r', \kappa '= \frac{ \kappa ''}{C_0}, 
\tilde N = R  N, 
\]
then we need
\[ \varepsilon^{1-2\zeta-\frac{1}{48}} \leq \left(\frac{1}{C_0}\cdot \frac{\kappa''}{32(1+\Vert \tilde A \Vert)}\varepsilon^{\zeta}\right)^2\]
or sufficiently
\[ \varepsilon^{1-2\zeta-\frac{1}{48}} \leq \left(\frac{1}{C_0}\cdot \frac{\kappa''}{32(1+\Vert A \Vert+\pi N)}\varepsilon^{\zeta}\right)^2\]
which holds true if
\[ \varepsilon^{1-2\zeta-\frac{1}{48}} \leq \left(\frac{1}{C_0}\cdot \frac{\kappa''}{32(2+\pi )}\varepsilon^{2\zeta}\right)^2\]

\noindent (where we have used the assumption that $\Psi\geq id$),
which holds true by assumption \ref{351-epsilon-small}.
Therefore we can apply lemma \ref{344} to get the maps $X \in U_{r'}(\T^d, sl(2, \R))$, $F' \in  U_{ r'}(\T^d, sl(2, \R))$, and a matrix $A' \in sl(2, \R)$ such that

\begin{itemize}
\item $A'$ has $BR_\omega ^{R N} (\frac{3\kappa ''}{4C_0})$ spectrum,
\item $\Vert A' - \tilde A \Vert \leq \tilde \varepsilon \leq \varepsilon^{\frac{23}{24}}$ (because $1-2\zeta -  \frac{1}{48} \geq \frac{23}{24} $),
which implies that
\[ \Vert A' - A \Vert \leq \Vert A' - \tilde A \Vert + \Vert A - \tilde A \Vert \leq \varepsilon^{\frac{23}{24}}+\pi  N \]
and thus
\[ \Vert A' \Vert \leq \Vert A \Vert + \varepsilon^{\frac{23}{24}} + \pi  N \]
which is property \ref{norme-A'},
\item $\partial_\omega e^X = (\tilde A + \tilde F)e^X - e^X(A' + F')$.

Let $\bar F' = \psi' F'(\psi' )^{-1}\in \mathcal{C}^0(\mathbb{T}^d,sl(2,\mathbb{R}))$
and $\bar A ' \in U_{r}(2 \T^d, sl(2, \R))$ such that
\[ \partial_\omega \psi'  = \bar A' \psi'  - \psi'  A' \]
(which means that $\bar A'$ is reducible to $A'$, hence Property \ref{A'red} with $\psi ' :=  \psi \Phi$). 
Then the function $Z' := \psi'  e^X (\psi')^{-1}\in \mathcal{C}^0(\mathbb{T}^d,SL(2,\mathbb{R}))$ is solution of 
\[ \partial_\omega Z' = (\bar A + \bar F)Z' - Z'(\bar A' + \bar F') \]
hence Property \ref{conjug}. This conjugation also implies that $\bar{A}'\in \mathcal{C}^0(\mathbb{T}^d,sl(2,\mathbb{R}))$.

\item if $\tilde A$ has two different eigenvalues,
since $\Phi$ is trivial with respect to $\LL_{A}$ which is identical to $\LL_{\tilde A}$, 
by Lemma \ref{344} and the expression \eqref{trivialPhi-step},
\[ \vert \Phi X \Phi^{-1}\vert_{  r'}\leq 
\frac{4C_0^{15}}{\kappa^{''13}}
\Psi( 3R N ) \vert \Phi \tilde F \Phi^{-1}\vert_{  r}\]
and
\[\vert \Phi F' \Phi^{-1} \vert_{  r'}  \leq \frac{4C_0^{15}}{\kappa^{''13}} e^{\vert \Phi X\Phi^{-1} \vert_{  r'}} \big(
\vert \Phi \tilde F \Phi^{-1} \vert_{  r} \big[ e^{-2\pi \Lambda(R N)(r-r')}   +  \vert \Phi \tilde F \Phi^{-1}\vert_{  r} \Psi(3RN)(2e+e^{\vert \Phi X \Phi^{-1} \vert_{  r'}})  \big]\]

\bigskip
otherwise if $\tilde A$ is nilpotent,
\[ \vert X\vert _{r'}\leq \frac{3}{\kappa^3}\Psi( 3R N )^3\vert  \tilde F \vert_{r}\]
and
\[ \vert F'\vert _{r'}  \leq \frac{3}{\kappa^3}
e^{\vert  X \vert_{  r'}}  \vert  \tilde F  \vert_{  r} \big[ e^{-2\pi \Lambda(R N)(r-r')} \qquad \]
\[  \qquad +  \vert  \tilde F \vert_{  r} \Psi(3RN)^3(2e+e^{\vert  X  \vert_{  r'}})  \big]\]
and if $ad_{\tilde{A}}=0$,
\[ \vert X\vert _{r'}\leq 
\frac{1}{\kappa}\Psi( 3R N )\vert  \tilde F \vert_{r}\]
and 
\[
\vert F'\vert _{r'}  \leq 
\frac{1}{\kappa}
e^{\vert  X \vert_{  r'}}  \vert  \tilde F  \vert_{  r} \big[ e^{-2\pi \Lambda(R N)(r-r')} \qquad \]
\[  \qquad +  \vert  \tilde F \vert_{  r} \Psi(3RN)(2e+e^{\vert  X  \vert_{  r'}})  \big].\]

\bigskip
Notice that in any case (since $\Phi\equiv I$ if $\tilde A$ is nilpotent or $\operatorname{ad}_{\tilde A}=0$), we have

\begin{equation}\label{estim-X-step} \vert \Phi X \Phi^{-1}\vert_{  r'}\leq 
\frac{4C_0^{15}}{\kappa^{''13}}
\Psi( 3R N )^3 \vert \Phi \tilde F \Phi^{-1}\vert_{  r}\end{equation}
and 
\begin{equation}\label{estim-F'-step} \vert \Phi F' \Phi^{-1} \vert_{  r'}  \leq 
\frac{4C_0^{15}}{\kappa^{''13}}
e^{\vert \Phi X\Phi^{-1} \vert_{  r'}} 
\vert \Phi \tilde F \Phi^{-1} \vert_{  r} \big[ e^{-2\pi \Lambda(R N)(r-r')} +  \vert \Phi \tilde F \Phi^{-1}\vert_{  r} \Psi(3RN)^3(2e+e^{\vert \Phi X \Phi^{-1} \vert_{  r'}})  \big]\end{equation}

\end{itemize}
\textbf{Estimates} \\

\textbf{Estimate of $\Psi',\Psi^{'-1}, A'$}

\bigskip

With the assumption $\varepsilon \leq  (2C_0)^{-96} \big( \frac{ \kappa''}{\Vert A \Vert +1} \big)^{576}$,
we have 
\[ \vert \Phi \vert_{ r} \leq  \varepsilon^{-\frac{1}{96}}e^{2\pi \Lambda(\frac{N}{2}) r}\]
and similarly for $\Phi^{-1}$.
Moreover since $\vert \psi \vert_{ r} \leq \varepsilon^{-\zeta}$, we get property \ref{estim-psi'}:
\[\vert \psi'\vert_{r} = \vert\psi \Phi\vert_{r} \leq \vert \psi \vert_{r} \vert \Phi \vert_{r} \leq \varepsilon^{-\zeta-\frac{1}{96}}e^{2\pi \Lambda(\frac{N}{2})r}\]
and similarly for $\psi'^{-1}$. Notice that this inequality remains true if $\Phi \equiv id$.

\bigskip
Notice that if $\Phi \not\equiv I$ (that is to say if the spectrum of $A$ is resonant), then from lemma \ref{renormalization} we get $\Vert \tilde A \Vert \leq \frac{1}{2}\kappa''$ and then for $\varepsilon \leq \varepsilon_0$ defined in lemma \ref{smallness-epsilon} (see equation \eqref{eq:cond1.1}),

\begin{equation*}\begin{split}
\Vert A' \Vert & \leq \Vert \tilde A \Vert + \Vert \hat{\tilde F}(0) \Vert \\
		    & \leq \frac{1}{2}\kappa \varepsilon^{\zeta} + \varepsilon^{1-2\zeta - \frac{1}{48}} \\
		    & \leq \frac{1}{2}\kappa \varepsilon^{\frac{1}{1728}} + \varepsilon^{\frac{845}{864}}\\
		    & \leq \frac{3}{4}\kappa \varepsilon^{\frac{1}{1728}} = \frac{3}{4}\kappa''
\end{split}\end{equation*}
and property \ref{norme2-A'} is satisfied.

\bigskip
\textbf{Estimate of $Z^{'\pm 1}-I,\psi^{-1}(Z^{'\pm 1})\psi$ and its derivative}

\bigskip
Since $\tilde F = (\psi \Phi)^{-1} \bar F \psi \Phi$, then 

\begin{equation}\label{estim-Ftilde} \vert \Phi\tilde F \Phi^{-1} \vert_{ r} = \vert \psi^{-1}\bar F \psi  \vert_{ r'} \leq \vert \bar F  \vert_{ r'} \varepsilon^{-2\zeta} = \varepsilon^{1-2\zeta}
\end{equation}

\bigskip

Recall the estimate \eqref{estim-X-step}:

\begin{equation}\vert \Phi X \Phi^{-1}\vert_{  r'}\leq 
\frac{4C_0^{15}}{\kappa''^{13}}
\Psi( 3R N )^3 \vert \Phi \tilde F \Phi^{-1}\vert_{  r}\end{equation}
therefore by \eqref{estim-Ftilde}, and for $\varepsilon \leq \varepsilon_0$ defined in lemma \ref{smallness-epsilon} ((see equation \eqref{eq:cond1.2}),
\begin{equation}\label{estim-X-bis} \vert \Phi X \Phi^{-1} \vert_{ r'} \leq \frac{4C_0^{15}}{\kappa^{13}}\varepsilon^{-16\zeta}\varepsilon^{1-2\zeta}\leq \varepsilon^{\frac{7}{8}} \end{equation}
then
\[ e^{ \vert \Phi X \Phi^{-1}\vert_{ r'} } \leq e^{\varepsilon^{\frac{7}{8}}} \leq 2 \]

\bigskip
\noindent We now estimate $\vert Z'-I\vert _{r'}=\vert \psi \Phi (e^X-I)(\psi \Phi)^{-1}\vert_{ r'}$. 
From \eqref{estim-X-bis},
\[ \vert \Phi e^X  \Phi^{-1} -Id \vert_{r'}\leq e\vert \Phi X\Phi^{-1}\vert _{r'}
\leq e \varepsilon^{\frac{7}{8}}\]

\noindent Then
\[ \vert Z'-I\vert _{r'}=\vert \psi \Phi e^X (\psi \Phi)^{-1}-Id\vert_{ r'} 
\leq \vert \psi \vert_{ r'}\vert \Phi e^X \Phi^{-1}-Id\vert_{ r'} \vert \psi^{-1}\vert_{ r'}\leq e\varepsilon^{\frac{7}{8}-2\zeta}\]

\noindent hence property \ref{estim-Z'} is satisfied.
If $\Phi \equiv I$, we have
\[\psi^{-1}Z'\psi =\psi^{-1}\psi\Phi e^X (\psi \Phi)^{-1}\psi = e^X, \]
therefore
\[\vert \psi^{-1}Z'\psi \vert_{r'} \leq \vert e^X \vert_{r'}\leq e^{\varepsilon^{\frac{7}{8}}} \]

\noindent which is the first part of the property \ref{estim-Z'-deriv}.
Now Lemma \ref{344} also states that if $\Phi\equiv I$ (that is, $\tilde A = A$),
\[ \vert \partial_\omega X\vert_{r'}\leq 2\Vert A\Vert \ \vert X\vert_{r'}+\vert \tilde F\vert_{r'}\]
which implies that

\[ \vert \partial_\omega X\vert_{r'}\leq 2\Vert A\Vert\varepsilon^{\frac{7}{8}}+\varepsilon^{1-2\zeta 
}\leq \varepsilon^{\frac{7}{8}}(2\Vert A \Vert+1)\]

\noindent
and by the assumption \ref{351-epsilon-small},
\[\vert \partial_\omega X\vert_{r'}\leq\varepsilon^{\frac{4}{5}}.\]
Therefore,
\[\vert \partial_\omega(\psi^{-1}Z'\psi) \vert_{r'}= \vert \partial_\omega(X)e^X \vert_{r'}\leq e^{\varepsilon^{\frac{7}{8}}}\varepsilon^{\frac{4}{5}}\leq \varepsilon^{\frac{1}{2}}\]
hence property \ref{estim-Z'-deriv}.

\bigskip
\textbf{Estimate of $\psi^{-1}\bar{F}'\psi=\Phi F'\Phi^{-1}$}

\bigskip
From Equation \eqref{estim-F'-step},

\[\vert \Phi F' \Phi^{-1} \vert_{ r'}  \leq 4C_0^2 e^{\vert \Phi X\Phi^{-1} \vert_{  r'}} \big( \frac{C_0}{ \kappa''}\big)^{13} \vert \Phi \tilde F \Phi^{-1} \vert_{ r} \big[ e^{-2\pi \Lambda(R  N)(r-r')} \qquad \]
\[  \qquad +  \vert \Phi \tilde F \Phi^{-1} \vert_{  r} \Psi(3R N)^3(2e+e^{\vert \Phi X \Phi^{-1} \vert_{ r'}})  \big]\]

Moreover, by definition, we have  $\Lambda(RN) = \frac{50 \delta \vert \log \varepsilon \vert}{\pi (r-r')} $, thus
\[e^{-2\pi \Lambda(R N)(r-r')} = \varepsilon^{100\delta} \] 
and then, because we assumed $\varepsilon \leq (2C_0)^{-96} \big( \frac{ \kappa''}{\Vert A \Vert +1} \big)^{576}$ and $\Psi(3 R  N)  = \varepsilon^{-\zeta} $, 
\[\vert \Phi F' \Phi^{-1} \vert_{  r'}  
\leq 8C_0^2\varepsilon^{-\frac{1}{96}}\varepsilon^{1-2\zeta}(\varepsilon^{100\delta}+ 8\Psi(3R N)^3 \varepsilon^{1-2\zeta}).\]

Thus 
\[\vert \Phi F' \Phi^{-1}\vert_{ r'}
\leq 8C_0^2\varepsilon^{1-2\zeta-\frac{1}{96}}(\varepsilon^{100\delta}+\varepsilon^{1-6\zeta})
\leq \varepsilon^{\frac{3}{2}}.  \]
Hence property \ref{estim-F'} holds for $\varepsilon\leq \varepsilon_0$ as defined in lemma \ref{smallness-epsilon} (see equation \eqref{eq:cond1.3}).
\newline

\end{proof}

\section{Inductive step}\label{inductivestep}

Let's define the following functions which will be used for the complete iterative step :
\[ \tag{P}\label{parameters}
 \left\{\begin{array}{c}
\displaystyle \kappa''(\varepsilon) =\kappa \varepsilon^{\zeta}\\
\displaystyle N(r,\varepsilon) = \Lambda^{-1}\big(\frac{50     \vert \log \varepsilon \vert}{\pi r}\big) \\
\displaystyle R(r,\varepsilon) =  \frac{1}{3N(r,\varepsilon)}\Psi^{-1}(\varepsilon^{-\zeta}) \\
\displaystyle r''(r,\varepsilon) = r-\frac{50 \delta \vert \log \varepsilon \vert}{\pi \Lambda(R(r,\varepsilon)N(r,\varepsilon))}
\end{array}\right.
\]

\noindent Note that these definitions match with lemma \ref{351}.

\begin{lemma}\label{352} Let
\begin{itemize}
\item $A \in sl(2, \R)$,
\item $r>0,$ 
\item $\bar A, \bar F \in U_{r}(\T^d, sl(2, \R)), \psi \in U_{r}(2\T^d, SL(2,\R))$,
\item $\vert \bar F \vert_r = \varepsilon$.
\end{itemize}
Suppose that

\begin{enumerate}
\item $\varepsilon\leq \varepsilon_0$, where $\varepsilon_0$ is defined in Lemma \ref{smallness-epsilon},
\item $r''>0$,
\item $\bar A$ is reducible to $A$ by $\psi$,
\item for all $G\in \mathcal{C}^0(\mathbb{T}^d,sl(2,\mathbb{R}))$, $\psi^{-1}G \psi \in \mathcal{C}^0(\T^d, sl(2,\R))$,
\item $\vert \psi ^{\pm 1}\vert_{r} \leq \varepsilon^{-\zeta}$,
\item \label{estimationA} $\Vert A \Vert \leq \varepsilon^{-\frac{\zeta}{2}}$,
\end{enumerate}
then, there exist
\begin{itemize}
\item $Z' \in U_{ r''}(\T^d, SL(2, \R))$,
\item  $\bar A', \bar F' \in U_{ r''}(\T^d, sl(2, \R))$,
\item $\psi' \in U_{ r}(2\T^d, SL(2, \R))$,
\item $A' \in sl(2,\R)$
\end{itemize}
satisfying the following properties:

\begin{enumerate}
\item \label{prop1} $\bar A'$ is reducible to $A'$ by $\psi '$,
\item \label{prop2} 
for all $G\in \mathcal{C}^0(\mathbb{T}^d,sl(2,\mathbb{R}))$, $\psi^{-1}G \psi \in \mathcal{C}^0(\T^d, sl(2,\R))$,
\item \label{prop3} $\vert \bar{F'} \vert_{ r''} \leq \varepsilon^{2\delta}$,
\item \label{prop4} $\vert \psi '^{\pm 1} \vert_{ r''} \leq \varepsilon^{-2\delta \zeta}$,
\item \label{prop5} $\Vert A' \Vert \leq  \Vert A \Vert +\varepsilon^{-\zeta}\leq \varepsilon^{-(2\delta)\frac{\zeta}{2}}$, 
\item \label{prop6} \[ \partial _\omega Z'  = (\bar A + \bar F) Z' - Z' (\bar A ' + \bar F'),\]
\item  \label{prop7} \[ \vert Z '^{ \pm 1} - Id \vert _{ r''} \leq \varepsilon^{\frac{9}{10}}. \]
\item \label{prop8} If moreover the spectrum of $A$ was not $BR_\omega^{R(r,\varepsilon)N(r, \varepsilon)}(\kappa''(\varepsilon))$, we actually have $\Vert A' \Vert \leq \kappa''(\varepsilon)$;
\item \label{prop9} If the spectrum of $A$ was $BR_\omega^{R(r,\varepsilon)N(r, \varepsilon)}(\kappa''(\varepsilon))$, we actually have $\psi' = \psi$ and  then 
\begin{equation}\label{prop9.1}\vert \psi^{-1}Z'^{\pm 1}\psi  \vert_{r''}\leq (1+2\varepsilon)e^{2\varepsilon} \end{equation}
and
\[\vert \partial_\omega(\psi^{-1}Z'^{\pm 1}\psi  )\vert_{r''}\leq \varepsilon^{\frac{1}{4}}. \]
\end{enumerate}
\end{lemma}

\begin{proof}
\textbf{Removing the resonances and first step} \\

Let 
$R=R(r,\varepsilon)$, $N=N(r, \varepsilon)$, $ \kappa'' =  \kappa''(r,\varepsilon)$, $r'' = r''(r, \varepsilon)$. 
 Since $\kappa''=\kappa\varepsilon^{\zeta}$, $\Vert A \Vert \leq \varepsilon^{-\frac{\zeta}{2}}$ and $\varepsilon \leq \varepsilon_0$ as defined in Lemma \ref{smallness-epsilon} (see equation \eqref{eq:cond2.0.0}),

\[ \varepsilon^{1-576\zeta} \leq  (2C_0)^{-96}\big(\frac{\kappa}{32(\Vert A \Vert +1)}\big)^{576} \]
therefore
\[ \varepsilon \leq (2C_0)^{-96}\Big(\frac{ \kappa '' }{32(\Vert A \Vert +1) }\Big)^{576}\]
and the assumption of lemma \ref{351} is satisfied.
We can apply lemma \ref{351} to get:
\begin{itemize}
\item $Z_1\in U_{ \frac{r+r''}{2}}(\T^d, SL(2, \R))$, 
\item $\psi'  \in U_{ \frac{r+r''}{2}}(2\T^d, SL(2, \R))$,
\item $A_1\in sl(2, \R)$,
\item $\bar A_1 \in U_{ \frac{r+r''}{2}}(\T^d, sl(2,\R))$,
\item and $F_1 = (\psi ')^{-1}\bar F_1 \psi' $,
with $\bar{F}_1\in U_{ \frac{r+r''}{2}}(\mathbb{T}^d,sl(2,\mathbb{R}))$
\end{itemize}
such that
\begin{enumerate}
\item $\bar A_1$ is reducible to $A_1$ by $\psi'$,
\item for all $G\in \mathcal{C}^0(\mathbb{T}^d,sl(2,\mathbb{R}))$, $\psi'^{-1}G\psi'\in \mathcal{C}^0(\mathbb{T}^d,sl(2,\mathbb{R}))$,
which implies that $F_1\in \mathcal{C}^0(\T^d, sl(2,\R))$,
\item $A_1$ has $BR_\omega^{R N}(\frac{3}{4}\frac{\kappa ''}{C_0})$ spectrum,
\item \[ \partial _\omega Z_1  = (\bar A + \bar F) Z_1 - Z_1 (\bar A_1 + \bar F_1),\]
\item 

\begin{equation}\label{A_1-wrt-A}\Vert A_1\Vert \leq \Vert A \Vert + \varepsilon^{\frac{23}{24}} + \pi N,\end{equation}
\item 
\begin{equation}\label{estim-Z_1-iteration}
 \vert Z_1^{\pm 1} - Id \vert _{ r''} \leq \varepsilon^{\frac{8}{9}}, \end{equation}
\item for all $s' >0$,
\[ \vert \psi'^{-1}\psi \vert_{ s'} \leq 2C_0 \big( \frac{1}{\kappa ''}\big)^{6}e^{2\pi \Lambda(\frac{N}{2}) s'} \] \[\vert \psi^{-1}\psi ' \vert_{ s'} \leq 2 C_0 \big( \frac{1}{\kappa ''}\big)^{6}e^{2\pi  \Lambda(\frac{N}{2}) s'}\] \\

\item $\vert \psi'^{\pm 1} \vert_{ r''} \leq \varepsilon^{-\zeta-\frac{1}{96}}e^{2\pi \Lambda(\frac{N}{2}) r}$,

\item \[\vert \psi^{-1} \bar F_1 \psi \vert_{r''} \leq \varepsilon^{\frac{5}{4}}, \]
\item If the spectrum of $A$ was not $BR_\omega^{RN}(\kappa'')$, $\Vert A_1 \Vert \leq \frac{1}{2}\kappa''$;
\item If the spectrum of $A$ was $BR_\omega^{RN}(\kappa'')$, we actually have $\psi' = \psi$ and  then 
\[ \vert \psi^{-1}Z_1^{\pm 1}\psi \vert_{r''} \leq e^{\varepsilon^{\frac{7}{8}}}\]
\begin{equation}\label{estim-derivZ1}\vert \partial_\omega(\psi^{-1}Z_1^{\pm 1}\psi)\vert_{r''} \leq \varepsilon^{\frac{1}{2}}.  \end{equation}
\end{enumerate}
\textbf{Second step : iteration without resonances} \\

We will now iterate lemma \ref{344} a certain number of times, without renormalization. \\
Let $l = E(\frac{\log(100\delta)}{\log(\frac{4}{3})}) = 56$ which satisfies 

\[ \varepsilon^{(\frac{4}{3})^{l+1}} \leq e^{-2\pi \Lambda(R N)(r-r'')} =\varepsilon^{100\delta}  \leq \varepsilon^{(\frac{4}{3})^l}.\]
Define for all $j\geq 0$, the sequences $\varepsilon'_j = \varepsilon^{(\frac{5}{4})^j}\varepsilon^{-\frac{1}{48}}$ and $r'_j= \frac{r+r''}{2} - j\frac{r-r''}{2l}$
. Thus $r'_0=\frac{r+r''}{2}$ and $r'_l=r''<r$.\\
We want to iterate $l-1$ times lemma \ref{344}, from $j=2$, with
\begin{itemize}
\item $\tilde \varepsilon = \varepsilon'_{j-1}$,
\item $\tilde r = r'_{j-2} $,
\item $\tilde r' = r'_{j-1} $,
\item $\kappa ' = (\frac{3}{4})^{j-1}\frac{ \kappa ''}{C_0}$,
\item $\tilde N = R N$,
\item $\tilde F = F_{j-1}$,
\item $\tilde A = A_{j-1}$,
\item $\Phi = \psi^{-1}\psi'$,
\end{itemize}

\underline{First iterate of lemma \ref{344}}:
From
\[\vert \psi^{-1}\bar F_1 \psi \vert_{0} \leq \varepsilon^{\frac{5}{4}}\]
and
 \[ \vert \psi'^{-1}\psi \vert_{0} \leq 2C_0 \big( \frac{1}{\kappa ''}\big)^{6}\leq \varepsilon^{-\frac{1}{96}},\ \vert \psi^{-1}\psi ' \vert_{0} \leq 2 C_0 \big( \frac{1}{\kappa ''}\big)^{6}
\leq \varepsilon^{-\frac{1}{96}},\]
then
\[ \Vert \hat{F_1}(0) \Vert \leq \vert\psi'^{-1} \psi \vert_0 \vert\psi^{-1} \psi' \vert_0 \vert \psi^{-1}\bar{F_1}\psi \vert_0 \leq 
\varepsilon^{\frac{5}{4}-\frac{1}{48}}.\]

As $A_1$ has $BR_\omega^{RN}(\frac{3}{4}\frac{\kappa''}{C_0})$ spectrum, to apply lemma \ref{344} we need
\[\varepsilon^{\frac{5}{4}-\frac{1}{48}} \leq \big(\frac{(\frac{3}{4})\frac{\kappa''}{C_0}}{32(1 + \Vert  A_1 \Vert)} \big)^{2}\frac{1}{\Psi(RN)^2}\]
and since
\[\Vert A_1 \Vert \leq \Vert A \Vert + \varepsilon^{\frac{23}{24}} + \pi N
\leq \varepsilon^{-\frac{\zeta}{2}} + \varepsilon^{\frac{23}{24}} + \pi \varepsilon^{-\zeta} \]

\noindent (this last inequality comes from the fact that $\Psi\geq id$, which implies that $N=\frac{1}{3R}\Psi^{-1}(\varepsilon^{-\zeta})\leq \varepsilon^{-\zeta}$),
this remains true if
\[\varepsilon^{\frac{5}{4}-\frac{1}{48}} \leq \big(\frac{(\frac{3}{4})\frac{\kappa''}{C_0}}{32(1 + (\pi +1)\varepsilon^{-\zeta}+\varepsilon^{\frac{23}{24}}) }\big)^{2}\frac{1}{\Psi(RN)^2}\]
which holds for $\varepsilon \leq \varepsilon_0$ as in lemma \ref{smallness-epsilon} (see equation \eqref{eq:cond2.0}).

\bigskip
\underline{Iteration of lemma \ref{344}}

If for some $j\geq2$
\[ \varepsilon^{(\frac{5}{4})^j-\frac{1}{48}} \leq 
\Big( \frac{(\frac{3}{4})^j\frac{ \kappa''}{C_0}}{
{32}(1+\varepsilon^{\frac{23}{24}} +(1+\pi)\varepsilon^{-
\zeta}  + 2\varepsilon^{\frac{5}{4}-\frac{1}{48}})}
\Big)^{2}
\frac{1}{\Psi(R N)^2},\]

which holds true for $\varepsilon \leq \varepsilon_0$ as in lemma \ref{smallness-epsilon} (see equation \eqref{eq:cond2.1}), then 
\[\varepsilon'_j \leq \Big( \frac{(\frac{3}{4})^j\frac{ \kappa''}{C_0}}{{32}(1+\Vert A_1 \Vert + \sum_{i=1}^{j-1}\varepsilon _i)}\Big)^{2}\frac{1}{\Psi(R N)^2}. \]

Let $j\geq 2$ and assume that $A_{j-1}$ has $BR_\omega^{R  N}((\frac{3}{4})^{j-1}\frac{\kappa''}{C_0})$ spectrum, $F_{j-1} \in U_{r'_{j-2}}(\T^d, sl(2,\R))$, and 
\[ \Vert \hat F_{j-1}(0) \Vert \leq \varepsilon'_{j-1};\ |\Psi^{-1}\Psi' F_{j-1}\Psi^{'-1}\Psi|_{r'_{j-2}}\leq \varepsilon^{(\frac{5}{4})^{j-1}}. \]
We obtain via lemma \ref{344} functions $F_j$, $X_j\in U_{r_{j-1}}(\mathbb{T}^d,sl(2,\mathbb{R}))$ and a matrix $A_j \in sl(2,\R)$ such that
\begin{enumerate}
\item $A_j$ has $BR_\omega ^{R  N}((\frac{3}{4})^j \frac{\kappa ''}{C_0})$ spectrum,
\item $\Vert A_{j} \Vert \leq \Vert A_{j-1}\Vert + \varepsilon'_{j-1}$,
\item \[\partial_\omega e^{X_j} = (A_{j-1}+F_{j-1})e^{X_j} - e^{X_j}(A_j +F_j), \]

\item the following estimates hold:
\begin{itemize}
    \item if $A_{j-1}$ has two different eigenvalues:
    \[\vert \psi^{-1}\psi' X_j \psi'^{-1}\psi \vert_{r'_{j-1}} \leq 4C_0^2 \big(\frac{1}{(\frac{3}{4})^{j-1}\frac{\kappa''}{C_0}} \big)^{13} \Psi(3RN) \vert \psi^{-1}\psi'F_j \psi'^{-1}\psi\vert_{r'_{j-1}}, \]
\begin{equation*}\begin{split} \vert \psi^{-1} \psi ' F_j \psi'^{-1}\psi \vert_{ r'_{j-1}}& \leq 4C_0^2\big( \frac{1}{(\frac{3}{4})^{j-1}\frac{ \kappa ''}{C_0}}\big)^{13}  e^{\vert \psi^{-1}\psi ' X_{j-1} \psi'^{-1}\psi\vert_{ r'_{j-1}}}   \vert \psi^{-1} \psi ' F_{j-1} \psi'^{-1}\psi \vert_{ r'_{j-2}} \\
&\big[ e^{-2 \pi \Lambda(R N)(r'_{j-2}-r'_{j-1})} + \vert \psi^{-1} \psi ' F_{j-1} \psi'^{-1}\psi \vert_{ r'_{j-2}}\Psi(3R N)(2e+ e^{\vert \psi^{-1}\psi ' X_{j-1} \psi'^{-1}\psi\vert_{ r'_{j-1}}}) \big] ;
\end{split}\end{equation*}

\item if $A_{j-1}$ is nilpotent:
 \[ \vert X_j \vert_{r'_{j-1}} \leq \frac{3}{\kappa^3}\Psi(3RN)^3\vert F_{j-1}\vert_{r'_{j-1}}, \]
\[ \vert F_j \vert_{r'_{j-1}} \leq \frac{3}{\kappa^3}e^{\vert X_{j-1}\vert_{r'_{j-1}}}\vert F_{j-1}\vert_{r'_{j-2}} \big[e^{-2\pi \Lambda(RN)(r'_{j-2}-r'_{j-1})}+\vert F_{j-1}\vert_{r'_{j-2}}\Psi(3RN)^3(2e+e^{\vert X_{j-1}\vert_{r'_{j-1}}}) \big].\]

\item if $ad_{A_{j-1}}=0$:
 \[ \vert X_j \vert_{r'_{j-1}} \leq \frac{1}{\kappa}\Psi(3RN)\vert F_{j-1}\vert_{r'_{j-1}}, \]
\[ \vert F_j \vert_{r'_{j-1}} \leq \frac{1}{\kappa}e^{\vert X_{j-1}\vert_{r'_{j-1}}}\vert F_{j-1}\vert_{r'_{j-2}} \big[e^{-2\pi \Lambda(RN)(r'_{j-2}-r'_{j-1})}+\vert F_{j-1}\vert_{r'_{j-2}}\Psi(3RN)(2e+e^{\vert X_{j-1}\vert_{r'_{j-1}}}) \big].\]
\end{itemize}
\end{enumerate}
\bigskip
Notice that in any case we have
\begin{equation}\vert \psi^{-1}\psi' X_j \psi'^{-1}\psi \vert_{r'_{j-1}} \leq 4C_0^2 \big(\frac{1}{(\frac{3}{4})^{j-1}\frac{\kappa''}{C_0}} \big)^{13} \Psi(3RN) \vert \psi^{-1}\psi'F_j \psi'^{-1}\psi\vert_{r'_{j-1}} \label{Xj-}, \end{equation}

\begin{equation}\begin{split} \vert \psi^{-1} \psi ' F_j \psi'^{-1}\psi \vert_{ r'_{j-1}}& \leq 4C_0^2 \big( \frac{1}{(\frac{3}{4})^{j-1}\frac{ \kappa ''}{C_0}}\big)^{13}  e^{\vert \psi^{-1}\psi ' X_{j-1} \psi'^{-1}\psi\vert_{ r'_{j-1}}}   \vert \psi^{-1} \psi ' F_{j-1} \psi'^{-1}\psi \vert_{ r'_{j-2}} \\
&\big[ e^{-2 \pi \Lambda(R N)(r'_{j-2}-r'_{j-1})} +8 \vert \psi^{-1} \psi ' F_{j-1} \psi'^{-1}\psi \vert_{ r'_{j-2}}\Psi(3R N)\big] 
\label{estimFj},
\end{split}\end{equation}
so we will use these estimates to iterate lemma \ref{344}.
\bigskip

Estimates $ \varepsilon \leq (2C_0)^{-96}\Big(\frac{ \kappa '' }{32(\Vert A \Vert +1) }\Big)^{576} $ and 
$\vert \psi^{-1}\psi'  F_{j-1} \psi'^{-1} \psi  \vert_{r_{j-2}}  \leq \varepsilon^{(\frac{5}{4})^{j-1}}$ give \[ e^{\vert \psi^{-1}\psi'X_{j-1}\psi'^{-1}\psi \vert_{r'_{j-2}}}  \leq 2 \]
Hence from \eqref{estimFj}

\[ \vert \psi^{-1}\psi'F_j\psi'^{-1}\psi\vert_{ r'_{j-1}}  \leq  64\cdot 4C_0^2 \Psi(3R N) \big( \frac{\Psi(3R N)}{(\frac{3}{4})^{j-1}\frac{\kappa}{C_0} }\big)^{13} \vert \psi^{-1}\psi' F_{j-1} \psi'^{-1}\psi \vert_{  r'_{j-2}} \big[ e^{-2\pi \Lambda(R N)(r'_{j-2}-r'_{j-1})} \] \[ \qquad +  \vert \psi^{-1}\psi'F_{j-1} \psi'^{-1}\psi \vert_{  r'_{j-2}}  \big] \]
\[\leq 256C_0^2 \Psi(3R N)^{14}\big( \frac{1}{(\frac{3}{4})^{j-1}\frac{\kappa}{C_0} }\big)^{13} \varepsilon^{(\frac{5}{4})^{j-1}} ( \varepsilon^{\frac{50\delta}{l}} +  \varepsilon^{(\frac{5}{4})^{j-1}}  ) \]
Since $\Psi(3RN) = \varepsilon^{-\zeta}$ and $\varepsilon\leq \varepsilon_0$ defined in lemma \ref{smallness-epsilon} ((see equation \eqref{eq:cond2.2}), 
\begin{equation}\label{Fj}\vert \psi^{-1}\psi'F_j\psi'^{-1}\psi\vert_{ r'_{j-1}} \leq \varepsilon^{(\frac{5}{4})^j} \end{equation}

We will now estimate $\Vert \hat F_j(0) \Vert$ to iterate lemma \ref{344}:
\[ \Vert \hat{F}_j (0) \Vert \leq \vert F_j \vert_{ 0} = \vert ( \psi'^{-1}\psi )\psi^{-1}\psi'F_j \psi'^{-1} \psi(\psi^{-1} \psi') \vert_{ 0} \leq \vert \psi^{-1} \psi' \vert_{ 0} \vert \psi'^{-1} \psi \vert_{ 0} \vert \psi^{-1}\psi'F_j \psi'^{-1} \psi\vert_{ r_{j-1}} ,  \]

\noindent
therefore
\[ \Vert \hat F_j(0)\Vert \leq   \vert \psi^{-1} \psi' \vert_{0} \vert \psi'^{-1} \psi \vert_{0}  \varepsilon^{(\frac{5}{4})^j}
\leq \varepsilon^{(\frac{5}{4})^j}\varepsilon^{-\frac{1}{48}}=\varepsilon'_j,\]
and we can iterate lemma \ref{344}, $l-1$ times. 

Equations \eqref{Fj} and \eqref{Xj-}
imply that 
\begin{equation}\label{Xj}\vert \psi^{-1}\psi'X_j\psi'^{-1}\psi \vert_{r_{j-1}}\leq \varepsilon'_j\end{equation}

and 

\[e^{\vert \psi^{-1}\psi'X_j\psi'^{-1}\psi \vert_{r_{j-1}}}\leq 2 \qquad .\]

\textbf{Conclusion :}

Let $Z = e^{X_2}...e^{X_{l+1}} \in U_{r''}(\T^d,SL(2,\R))$.

Let $Z' = Z_1 \psi' Z \psi'^{-1}$, $A'=A_{l+1}$, $F'=F_{l+1}$, $\bar F' = \psi' F' \psi^{'-1}$ (hence Property \ref{prop2}) and $\bar A'$ such that
\[ \partial_\omega \psi' = \bar A' \psi ' - \psi' A', \]
then 
\[ \partial_\omega Z' = (\bar A + \bar F) Z'  - Z' (\bar A' + \bar F'), \]
hence the properties \ref{prop1} and \ref{prop6} hold. 

\bigskip
\noindent
We have
\[ \partial_\omega Z = (A_1 + F_1)Z - Z(A_{l+1}+F_{l+1})\]
and since 
for all $j\geq 2$, we have

\begin{equation}\label{major-Aj-iteration} \Vert A_j \Vert \leq \Vert A_{j-1}\Vert + \varepsilon^{(\frac{5}{4})^{j-1}}\varepsilon^{-\frac{1}{48}} \leq \Vert A_1 \Vert + \sum_{i=1}^{j-1}\varepsilon^{(\frac{5}{4})^{i}}\varepsilon^{-\frac{1}{48}}= \Vert A_1 \Vert + \sum_{i=1}^{j-1}\varepsilon'_i,
\end{equation}

\noindent
then

\[ \Vert A' \Vert \leq \Vert A_1 \Vert + \sum_{i=1}^l \varepsilon_i \leq \Vert A \Vert + \varepsilon^{\frac{23}{24}} + \pi  N + \sum_{i=1}^l \varepsilon^{(\frac{5}{4})^i-\frac{1}{48}}. \]
Remind that $\Psi\geq id$ implies
\[N \leq \varepsilon^{-\frac{\zeta}{2}} \]
and then, since $\Vert A \Vert \leq \varepsilon^{\frac{-\zeta}{2}}$,

\[ \Vert A' \Vert \leq \Vert A \Vert +\varepsilon^{-\zeta} 
\leq 2\varepsilon^{-\zeta} \leq \varepsilon^{-\delta\zeta} \]
thus the property \ref{prop5} holds if $\varepsilon \leq \varepsilon_0$ as defined in lemma \ref{smallness-epsilon} (see equation \eqref{eq:cond2.3}). 

\bigskip
Moreover,

\begin{equation}\label{estimFl+1} \vert \psi^{-1}\psi'F_{l+1}\psi'^{-1}\psi\vert_{ r'_{l}} \leq  \varepsilon^{(\frac{5}{4})^{l+1}} \end{equation}
and since $l=56$, one has 

\[ \vert \psi'F_{l+1}\psi'^{-1}\vert_{ r'_{l}} \leq |\psi|_{r'_l}|\psi^{-1}|_{r'_l}
\varepsilon^{(\frac{5}{4})^{l+1}}\leq \varepsilon^{(\frac{5}{4})^{57}-2\zeta}
\leq \varepsilon^{2\delta} 
\]

\noindent thus the property \ref{prop3} holds.
In the case the spectrum of $A$ was resonant, the function $\Phi$ used in lemma \ref{351} is not the identity and we have
\[ \Vert A' \Vert \leq \Vert A_1 \Vert + \sum_{i=1}^l \varepsilon'_i \leq \frac{1}{2}\kappa''(r, \varepsilon) + 2 \varepsilon'_1 \leq \frac{1}{2}\kappa \varepsilon^{\zeta} + 2\varepsilon^{\frac{5}{4}-\frac{1}{48}}  \leq \kappa \varepsilon^{\zeta} = \kappa''(r,\varepsilon) \]
since $\varepsilon \leq \varepsilon_0$ as defined in lemma \ref{smallness-epsilon} (see equation \eqref{eq:cond2.4}), whence the property \ref{prop8}.

\bigskip
\textbf{Estimates}

Now we will show property \ref{prop4} : $\vert\psi'^{\pm 1}\vert_{r''} \leq \varepsilon^{-
2\delta\zeta}$ .\\
We know that $\displaystyle \vert \psi'^{\pm 1} \vert_{ r''} \leq \varepsilon^{-\zeta-\frac{1}{96}}e^{2\pi \Lambda(\frac{N}{2}) r}$.
But, by definition of $\Lambda(N) = \frac{50\vert\log\varepsilon\vert}{\pi r}$,

\[e^{2\pi\Lambda(\frac{N}{2})r} \leq e^{2\pi\Lambda(N)r} \leq e^{100\vert \log \varepsilon \vert} = \varepsilon^{-100}  \]
therefore 

\[ \vert\psi'^{\pm 1}\vert_{r''} \leq \varepsilon^{-\zeta - \frac{1}{96}-100} \leq \varepsilon^{-2\zeta \delta} .\]

\noindent which is Property \ref{prop4}. Now we will show the property \ref{prop3}. One has
\[\vert \bar{F'}\vert_{r''} = \vert \psi'F'\psi'^{-1}\vert_{r''} = \vert \psi \psi^{-1}\psi' F' \psi'^{-1}\psi \psi^{-1}\vert_{r''} \leq \vert \psi \vert_r \vert \psi^{-1}\vert_r \vert \psi^{-1}\psi' F' \psi'^{-1}\psi \vert_{r''}\leq \varepsilon^{-2\zeta}\varepsilon^{(\frac{5}{4})^l} \leq \varepsilon^{2\delta} \]
where the last inequality uses equation \eqref{estimFl+1}, which gives property \ref{prop3}.

\bigskip

According to the estimate
\[ \vert Z_1^{\pm 1} - Id \vert _{ r''} \leq \varepsilon^{\frac{8}{9}}\]

\noindent obtained in \eqref{estim-Z_1-iteration},
we get
\begin{equation*}
\begin{split}
    \vert Z' - Id \vert_{  r''} & \leq  \vert Z_1 - Id \vert_{  r'_1}  + \vert Z_1\vert_{r'_1}  \vert \psi \vert_{  r}\vert \psi^{-1} \vert_{ r}\sum_{j=2}^{l+1} \vert \psi^{-1} \psi' X_j \psi'^{-1} \psi \vert_{  r'_j} \\
    & \leq \varepsilon^{\frac{8}{9}}+(\varepsilon^{\frac{8}{9}}+1)\varepsilon^{-2\zeta}\sum_{j=2}^{l+1}\vert \psi^{-1}\psi'X_j\psi'^{-1}\psi \vert_{r'_j} \\
\end{split} \end{equation*}

Therefore, by the estimate \eqref{Xj} and by definition of $l=E(\frac{\log(100\delta)}{\log(\frac{7}{6})})$,
\[  \vert Z' - Id \vert_{r''} \leq \varepsilon^{\frac{8}{9}} +(\varepsilon^{\frac{8}{9}}+1)
2\varepsilon'_1\varepsilon^{-2\zeta}
\leq 
\varepsilon^{\frac{9}{10}}
\]
hence \ref{prop7}. \\
 
\textbf{Proof of the property \ref{prop9}}

We now have to estimate $\psi^{-1}Z'\psi$ and its directional derivative in the case $A$ has a $BR_\omega^{R(r,\varepsilon)N(r,\varepsilon)}(\kappa''(\varepsilon))$ spectrum. 
In this case $\Phi \equiv I$ and $\psi = \psi'$, therefore 
$\psi^{-1}Z'\psi = \psi^{-1}Z_1\psi Z$. Therefore
\[\vert \psi^{-1}Z'\psi  \vert_{r''} \leq \vert \psi^{-1}Z_1\psi\vert_{r''} \vert Z\vert_{r''} \leq (1+2\varepsilon) \vert Z \vert_{r''}. \]

\noindent (where the last inequality comes from \eqref{estim-Z_1-iteration}).
Moreover,
\[ \vert Z \vert_{r''} = \vert \Pi_{k=2}^{l+1} e^{X_k} \vert_{r''}. \]
Now for all $k \in \llbracket 2, l \rrbracket$,
we have seen in \eqref{Xj} that $|X_k|_{r''}\leq \varepsilon'_k$.

\noindent
Therefore
\[ \vert  Z \vert_{r''} \leq   \vert \Pi_{k=2}^{l+1}e^{X_k} \vert_{r''} \leq e^{\sum_{k=2}^{l+1} 
\varepsilon'_k} \leq e^{2\varepsilon}\]
and finally,
\[ \vert \psi^{-1} Z' \psi \vert_{r''} 
\leq (1+2\varepsilon)e^{2\varepsilon}.\]

The estimate of $\psi^{-1}Z'^{-1}\psi$ is obtained in a similar way. This gives the property \eqref{prop9.1}.

Moreover,
\[ \vert \partial_{\omega}(\psi^{-1}Z' \psi) \vert_{r''} \leq \vert  \partial_\omega (\psi^{-1}Z_1\psi)Z \vert_{r''} +\vert \psi^{-1}Z_1\psi\partial_\omega(Z) \vert_{r''} \leq \varepsilon^{\frac{1}{2}} \vert Z \vert_{r''} + (1+2\varepsilon)\vert \partial_\omega(Z) \vert_{r''} \] 
where the last inequality comes from \eqref{estim-derivZ1}.
Now
\[\vert \partial_\omega(Z)\vert_{r''} 
\leq \sum_{k=2}^{l+1} \vert \partial_\omega X_k\vert_{r''} \prod_{j=2}^{l+1} e^{\vert X_j\vert _{r''}}.
\]
For all $k \in \llbracket 2, l+1 \rrbracket$, by construction of $X_k$,

\[ \vert \partial_\omega X_k \vert_{r''}\leq 2 \Vert A_k \Vert \vert X_k \vert_{r''} + \vert F_{k}\vert_{r''} \leq 2\Vert A_k \Vert \vert X_k \vert_{r''} + \varepsilon^{(\frac{5}{4})^{k}-\frac{1}{48}}. \]
Now for all $k \in \llbracket 2, l+1 \rrbracket$,
by the estimate \eqref{major-Aj-iteration}, the estimate \eqref{A_1-wrt-A} and condition \ref{normeA} of this lemma, for $\varepsilon \leq \varepsilon_0$ given by lemma \ref{smallness-epsilon} (see equation \eqref{eq:cond2.5.2}),

\[ \Vert A_k \Vert  \leq \varepsilon^{-\frac{\zeta}{2}} + \varepsilon^{\frac{23}{24}} + \pi \varepsilon^{-
\zeta} \leq \varepsilon^{-2\zeta}\]
therefore for $\varepsilon \leq \varepsilon_0$ given by lemma \ref{smallness-epsilon} (see equation \eqref{eq:cond2.6}),

\begin{equation*}\begin{split}
\sum_{k=2}^{l=1}\vert \partial_\omega X_k \vert_{r''} & \leq 2 \sum_{k=2}^{l+1}\Vert A_k \Vert \vert X_k \vert_{r''} + \sum_{k=2}^{l+1}\vert F_{k}\vert_{r''} \\
& \leq 2\sum_{k=2}^{l+1} \varepsilon^{-2\zeta}
\varepsilon'_k+ \sum_{k=2}^{l+1}\varepsilon'_k
\\
& \leq \varepsilon
\end{split}\end{equation*}
and finally, for $\varepsilon\leq \varepsilon_0$ like in lemma \ref{smallness-epsilon} (see equation \eqref{eq:cond2.7}),

\[ \vert \partial_{\omega}(\psi^{-1}Z' \psi) \vert_{r''} \leq 2\varepsilon^{\frac{1}{2}}+2\varepsilon^{\frac{7}{8}} \leq \varepsilon^{\frac{1}{4}}. \]

The estimate of $ \partial_{\omega}(\psi^{-1}Z^{'-1} \psi)$ is similar, which gives property \ref{prop9}.
\end{proof}

\section{Almost reducibility}

Here we complete the proof of the main theorem.

\begin{theorem}\label{theoreme}
Let $r_0 >0 $, $A\in sl(2, \R)$ and $F \in U_{r_0}(\T^d, sl(2, \R))$. Then, if
\[ \vert F \vert_{r} \leq \varepsilon_0 \]

\noindent where $\varepsilon_0$ satisfies the assumptions above, and
\[\Vert A \Vert \leq \varepsilon_0^{-\frac{\zeta}{2}}, \]
then for all $\varepsilon \leq \varepsilon_0$, there exist

\begin{itemize}
\item $r_\varepsilon >0$, $k_\varepsilon \in \N$,
\item $Z_\varepsilon \in U_{r_\varepsilon}( \T^d, SL(2, \R))$,
\item $A_\varepsilon \in sl(2, \R)$,
\item $\bar A_\varepsilon, \bar F_\varepsilon \in U_{r_\varepsilon}(\T^d, sl(2, \R))$,
\item $\psi_\varepsilon \in U_{r_{\varepsilon}}(2\T^d, SL(2, \R))$,
\end{itemize}

such that
\begin{enumerate}
\item \label{thm-prop1} $\bar A_\varepsilon$ is reducible to $A_\varepsilon$ by $\psi_\varepsilon$, with $\vert \psi_{k_\varepsilon} \vert_{r_\varepsilon} \leq \varepsilon^{-\zeta}$,
\item \label{thm-prop2} $\vert \bar F_\varepsilon \vert_{r_\varepsilon} \leq \varepsilon$,
\item \label{thm-prop3} for all $\theta \in \T^d$,
\[ \partial_\omega Z_\varepsilon (\theta) = (A + F(\theta)) Z_\varepsilon (\theta) - Z_\varepsilon (\theta)(\bar A_\varepsilon (\theta) + \bar F_\varepsilon (\theta))\]

\item \label{thm-prop4} \[ \vert Z_\varepsilon ^{\pm 1} - Id \vert_{r_\varepsilon} \leq 
\varepsilon_0^{\frac{9}{10}}.\]
\end{enumerate}
Moreover, either $\vert \partial_\omega Z_\varepsilon \vert_{r_\varepsilon}$ is bounded as $\varepsilon\rightarrow 0$ and $A+F$ is a reducible cocycle in $U_{r_\infty}(\T^d, sl(2,\R))$ for some $r_\infty > 0$, or for all $\varepsilon\leq \varepsilon_0$ there exists $\varepsilon'\leq \varepsilon$ such that \[ \Vert A_{\varepsilon'} \Vert \leq \kappa \varepsilon'^{\zeta}.\]
\end{theorem}

\begin{proof}
Remind parameters (\ref{parameters}) defined in section \ref{inductivestep} and define, for all $k \in \N,k\geq 1$,
\[
\varepsilon_k:=\varepsilon_0^{(2\delta)^k};\quad 
r_k := r_0- \displaystyle \sum_{i=0}^{k-1}\displaystyle \frac{50 \delta\vert \log \varepsilon_{i} 	\vert }{\pi  \Lambda(R(r_{i},\varepsilon_{i})N(r_{i},\varepsilon_{i})) }.\]
Notice that, by Lemma \ref{rk-limite-positive}, under assumption \ref{assumption-2bis}, for all $k\in \N$, $r_k >0$.

We can apply a first time lemma \ref{352}. There exist 

\begin{itemize}
\item $Z_1 \in U_{r_1}(\T^d, SL(2, \R))$,
\item $\bar A_1$, $\bar F_1 \in U_{ r_1}(\T^d, sl(2, \R))$,
\item $A_1 \in sl(2, \R)$,
\item $\psi_0 \in U_{r_1}(2\T^d, SL(2, \R))$,
\end{itemize}
such that
\begin{itemize}
\item $\bar A_1$ is reducible to $A_1$ by $\psi_0$,
\item for all $G\in \mathcal{C}^0(\mathbb{T}^d,sl(2,\mathbb{R}))$,
$\psi_0^{-1}G\psi_0 \in \mathcal{C}^0(\T^d, sl(2,\R))$,
\item $\vert \bar F_1 \vert_{r_1} \leq \varepsilon_1$,
\item $\vert \psi_0^{\pm 1} \vert_{r_1} \leq \varepsilon_1^{-\zeta}$,
\item $\Vert A_1 \Vert \leq 
\varepsilon_1^{-\frac{\zeta}{2}} $,
\item for all $\theta \in \T^d$  

\[ \partial_\omega Z_1(\theta) = (A + F(\theta))Z_1(\theta) - Z_1(\theta)(\bar A_1(\theta) + \bar F_1(\theta)), \]
\item 
\[  \vert Z_1^{\pm 1} - Id \vert_{r_1} \leq \varepsilon_0^{\frac{9}{10}},\] 

\item if moreover $A$ had a $BR_\omega^{R(r_0,\varepsilon_0)N(r_0,\varepsilon_0)}(\kappa''(\varepsilon_0))$ spectrum,

\begin{equation}\vert \psi_{0}^{-1}Z_{1}\psi_{0} \vert_{r_{1}} \leq (1+2\varepsilon_0)e^{2\varepsilon_0}\end{equation}
and if not, 
\[\Vert A_{1} \Vert \leq \kappa''(\varepsilon_0)\]
and
\begin{equation} \vert \partial_\omega(\psi_{0}^{-1}Z_{1}\psi_{0}) \vert_{r_{1}} \leq \varepsilon_{0}^{\frac{1}{4}}. \end{equation}
\end{itemize}

\underline{Iterative step :} let $k \geq 1$ and
\begin{itemize}
\item $ \bar A_k, \bar F_k \in U_{r_k}(\T^d, sl(2,\R))$,
\item $A_k \in sl(2, \R)$,
\item $\psi_{k-1} \in U_{r_k}(2\T^d, SL(2,\R))$,
\end{itemize}
such that
\begin{itemize}
\item $\bar A_k$ is reducible to $A_k$ by $\psi_{k-1}$,
\item for all $G\in \mathcal{C}^0(\T^d,sl(2,\mathbb{R}))$,
$\psi_{k-1}^{-1}G\psi_{k-1} \in \mathcal{C}^0(\T^d, sl(2,\R))$,
\item $\vert \bar F_k \vert_{r_k} \leq \varepsilon_k$,
\item $\vert \psi_{k-1}^{\pm 1} \vert_{r_k} \leq \varepsilon_k^{-\zeta}$,
\item $\Vert A_k\Vert \leq \varepsilon_k^{-\frac{\zeta}{2}}$.
\end{itemize}

We can one again apply lemma \ref{352} to get

\begin{itemize}
\item $Z_{k+1} \in U_{r_{k+1}}( \T^d, SL(2, \R))$,
\item $\bar A_{k+1}$, $\bar F_{k+1} \in U_{ r_{k+1}}(\T^d, sl(2, \R))$,
\item $A_{k+1} \in sl(2, \R)$,
\item $\psi_k \in U_{r_{k+1}}(2\T^d, SL(2, \R))$,
\end{itemize}
such that
\begin{itemize}
\item $\bar A_{k+1}$ is reducible to $A_{k+1}$ by $\psi_k$,
\item for all $G\in \mathcal{C}^0(\T^d,sl(2,\mathbb{R}))$,
$\psi_k^{-1}G\psi_k \in \mathcal{C}^0(\T^d, sl(2,\R))$,
\item $\vert \bar F_{k+1}\vert_{r_{k+1}} \leq \varepsilon_{k+1}$,
\item $\vert \psi_k^{\pm 1} \vert_{r_{k+1}} \leq \varepsilon_{k+1}^{-\zeta}$ ,
\item $\Vert A_{k+1} \Vert \leq \varepsilon_{k+1}^{-\frac{\zeta}{2}}$,
\item for all $\theta \in \T^d$, \[ \partial_\omega Z_{k+1}(\theta) = (\bar A_k(\theta) + \bar F_k(\theta))Z_1(\theta) - Z_{k+1}(\theta)(\bar A_{k+1}(\theta) + \bar F_{k+1}(\theta)), \]
\item 
\[  \vert Z_{k+1}^{\pm 1} - Id \vert_{ r_{k+1}} \leq \varepsilon_k^{\frac{9}{10}}\]

\item if moreover $A_k$ had a $BR_\omega^{R(r_k,\varepsilon_k)N(r_k,\varepsilon_k)}(\kappa''(\varepsilon_k))$ spectrum,

\begin{equation}\label{bound-Z_k-nonres}\vert \psi_{k+1}^{-1}Z_{k+1}\psi_{k+1} \vert_{r_{k+1}} \leq (1+2\varepsilon_k)e^{2\varepsilon_k}\end{equation}

and if not, 

$$||A_{k+1}||\leq \kappa''(\varepsilon_k)$$

and

\begin{equation} \label{bound-partial_omegaZ_k}\vert \partial_\omega(\psi_{k+1}^{-1}Z_{k+1}\psi_{k+1}) \vert_{r_{k+1}} \leq \varepsilon_{k}^{\frac{1}{4}}. \end{equation}
\end{itemize}

\underline{Result:} Let $\varepsilon \leq \varepsilon_0$ and $k_\varepsilon \in \N$ such that $\vert F \vert_r^{(2\delta)^{k_{\varepsilon}}}\leq \varepsilon$. Let
\[ \left\{\begin{array}{c}Z_\varepsilon = Z_1 \cdots Z_{k_\varepsilon} \\ \bar A_\varepsilon = \bar A_{k_\varepsilon} \\ \bar F_\varepsilon = \bar F_{k_\varepsilon} \\ \psi_\varepsilon =  \psi_{k_\varepsilon} \\ r_\varepsilon =  r_{k_\varepsilon} \end{array}\right. \]
then the properties \ref{thm-prop1} and \ref{thm-prop2} hold. Moreover for all $\theta \in \T^d$,
\[ \partial_\omega Z_{\varepsilon}(\theta) = (A + F(\theta))Z_\varepsilon(\theta) - Z_\varepsilon(\theta)(\bar A_\varepsilon(\theta) + \bar F_\varepsilon(\theta))\]
and the property \ref{thm-prop3} holds. Notice that, for all $k' > \kappa_{\varepsilon}$, if $\Vert A_{k_\varepsilon} \Vert \leq \kappa \varepsilon^{\zeta}$ (which is satisfied if, for example, the matrix $A_{k_{\varepsilon}-1}$ was resonant), then

\[ \Vert A_{k'} \Vert \leq \Vert A_{k_\varepsilon} \Vert + \sum_{i=k_\varepsilon+1}^{k'}\varepsilon_i\leq  \kappa''(\varepsilon) + 2 \varepsilon \leq 2\kappa''(\varepsilon).\]

We also have
\[\vert Z_1^{\pm 1} - Id \vert_{r_{\varepsilon}}  \leq \varepsilon_0^{\frac{9}{10}}
\Rightarrow \vert Z_1 \vert_{r_{\varepsilon}}  \leq 1 + \varepsilon_0^{\frac{9}{10}}\]
Let $k \in \N$ and suppose that for all $j \leq k-1$,
\[ \vert Z_1 \dots Z_j \vert_{r_\varepsilon}  \leq 2\]
then
\[ c_k :=  \vert Z_1 \dots Z_k - Id \vert_{r_\varepsilon} \leq \vert Z_{k-1} - Id \vert_{r_\varepsilon} \vert Z_1 \dots Z_{k-1} \vert_{r_\varepsilon} + \vert Z_1 \dots Z_{k-1} - Id \vert_{r_\varepsilon}  \]
\[\leq 2\varepsilon_{k-2}^{\frac{9}{10}}+c_{k-1}\]

\noindent
which implies 

\[c_k\leq 2\sum_{i=0}^{k-2}\varepsilon_i^{\frac{9}{10}}v\leq 4\varepsilon_0^{\frac{9}{10}}.\]

\noindent Finally 
\[ \vert (Z_1 \dots Z_k)^{\pm 1} -I\vert_{r_\varepsilon} \leq  \varepsilon_0^{\frac{9}{10}}\]
hence the property \ref{thm-prop4} holds. 
\bigskip
\textbf{Reducible case}

\bigskip

Suppose that there exists $\bar{k}$ such that for all $k'\geq \bar{k}$, $\psi_{k'} \equiv \psi_{\bar{k}}$ (which means that for all $k'\geq \bar{k}$, $A_{k'}$ has a $BR_\omega^{N(r_{k'},\varepsilon_{k'})R(r_{k'},\varepsilon_{k'})}(\kappa''(\varepsilon_{k'}))$ spectrum). Then

\begin{equation}\begin{split} \partial_\omega Z_\varepsilon & = \partial_\omega (\prod_{i=1}^{k_\varepsilon}Z_i) \\
 & = \partial_\omega(\prod_{i=1}^{\bar{k}-1} Z_i)(\prod_{j=\bar{k}}^{k_\varepsilon}Z_j) + (\prod_{i=1}^{\bar{k}-1}Z_i) \partial_\omega(\prod_{j=\bar{k}}^{k_\varepsilon}Z_j)\\
 & = \partial_\omega(\prod_{i=1}^{\bar{k}-1} Z_i)(\prod_{j=\bar{k}}^{k_\varepsilon}Z_j) + (\prod_{i=1}^{\bar{k}-1}Z_i) \partial_\omega(\prod_{j=\bar{k}}^{k_\varepsilon}\psi_{\bar{k}}\psi^{-1}_j Z_j \psi_j \psi_{\bar{k}}^{-1}) \\
 & = \partial_\omega(\prod_{i=1}^{\bar{k}-1} Z_i)(\prod_{j=\bar{k}}^{k_\varepsilon}Z_j) + (\prod_{i=1}^{\bar{k}-1}Z_i)[ \partial_\omega (\psi_{\bar{k}})\prod_{j=\bar{k}-1}^{k_\varepsilon}\psi_j^{-1}Z_j\psi_j \psi_{\bar{k}}^{-1} + \\
 & \qquad \qquad \qquad \qquad \psi_{\bar{k}} \partial_\omega (\prod_{j=\bar{k}}^{k_\varepsilon}\psi_j^{-1}Z_j\psi_j) \psi_{\bar{k}}^{-1} + \psi_{\bar{k}} \prod_{j=\bar{k}}^{k_\varepsilon}(\psi_j^{-1}Z_j\psi_j) \partial_\omega (\psi_{\bar{k}}^{-1})],
\end{split}\end{equation}
thus 
\begin{equation*}\begin{split}
\vert \partial_\omega Z_\varepsilon \vert_{r_\varepsilon} & \leq \vert \partial_\omega(\prod_{i=1}^{\bar{k}-1} Z_i )\vert_{r_\varepsilon} \vert \prod_{j=\bar{k}}^{k_\varepsilon}Z_j \vert_{r_\varepsilon} + \vert\prod_{i=1}^{\bar{k}-1}Z_i\vert_{r_\varepsilon} \vert\psi_{\bar{k}}\vert_{r_\varepsilon}\vert\partial_\omega(\prod_{j=\bar{k}}^{k_\varepsilon} \psi^{-1}_j Z_j\psi_j) \vert_{r_\varepsilon} \vert \psi_{\bar{k}}^{-1}\vert_{r_\varepsilon} \\
& + \vert\prod_{i=1}^{\bar{k}-1}Z_i\vert_{r_\varepsilon}\vert(\prod_{j=\bar{k}}^{k_\varepsilon}\psi^{-1}_j Z_j\psi_j) \vert_{r_\varepsilon}\vert \psi_{\bar{k}}^{-1}\vert_{r_\varepsilon}\vert\partial_\omega \psi_{\bar{k}}\vert_{r_\varepsilon} \\
&+\vert\prod_{i=1}^{\bar{k}-1}Z_i\vert_{r_\varepsilon}\vert(\prod_{j=\bar{k}}^{k_\varepsilon}\psi^{-1}_j Z_j \psi_j) \vert_{r_\varepsilon}\vert \partial_\omega \psi_{\bar{k}}^{-1}\vert_{r_\varepsilon} \vert\psi_{\bar{k}}\vert_{r_\varepsilon}.
\end{split}\end{equation*}

\noindent Since the factors 
$\vert\prod_{i=1}^{\bar{k}-1}Z_i\vert_{r_\varepsilon},\vert\partial_\omega \prod_{i=1}^{\bar{k}-1}Z_i\vert_{r_\varepsilon},\vert(\prod_{j=\bar{k}}^{k_\varepsilon}Z_j) \vert_{r_\varepsilon},\vert\psi_{\bar{k}}\vert_{r_\varepsilon},\vert\psi_{\bar{k}}^{-1}\vert_{r_\varepsilon},\vert\partial_\omega \psi_{\bar{k}}\vert_{r_\varepsilon},\vert\partial_\omega \psi_{\bar{k}}^{-1}\vert_{r_\varepsilon}$ are bounded uniformly in $\varepsilon$ (here we use \eqref{bound-Z_k-nonres}), there exist $K_1,K_2\geq 0$ independent of $\varepsilon$ such that

\[\vert \partial_\omega Z_\varepsilon\vert _{r_\varepsilon}\leq K_1+K_2 \vert\partial_\omega(\prod_{j=\bar{k}}^{k_\varepsilon} \psi^{-1}_jZ_j\psi_j) \vert_{r_\varepsilon}.\]

Moreover, by \eqref{bound-partial_omegaZ_k} and \eqref{bound-Z_k-nonres},
\begin{equation*}\begin{split}
\vert\partial_\omega(\prod_{j=\bar{k}}^{k_\varepsilon} \psi^{-1}_jZ_j\psi_j) \vert_{r_\varepsilon} 
& \leq \sum_{j=\bar{k}}^{k_\varepsilon}  \vert \partial_\omega( \psi_j^{-1}Z_j\psi_j)\vert_{r_\varepsilon}
    \prod_{\substack{\bar{k}\leq  i 
      \leq k_\varepsilon \\ i \neq j}}  \vert \psi_i^{-1}Z_i \psi_i \vert_{r_\varepsilon} \\ 
    & \leq \sum_{j = \bar{k}}^{k_\varepsilon}\varepsilon_j^{\frac{1}{4}}\prod_{\substack{\bar{k}\leq  i \leq k_\varepsilon \\ i \neq j}} (1 + 2e^{\varepsilon_i})e^{2\varepsilon_i} \end{split}\end{equation*}
    
    therefore

    \begin{equation*}\begin{split}
\vert\partial_\omega(\prod_{j=\bar{k}}^{k_\varepsilon} \psi^{-1}_jZ_j\psi_j) \vert_{r_\varepsilon} 
&  \leq 2 \sum_{j = \bar{k}}^{k_\varepsilon} \varepsilon_j^{\frac{1}{4}}e^{2\varepsilon_{\bar{k}}} \\
    & \leq 8\varepsilon_{\bar{k}}^{\frac{1}{4}}e^{2\varepsilon_{\bar{k}}} \\
    & \leq 16 \varepsilon_{\bar{k}}^{\frac{1}{4}}
\end{split}\end{equation*}
and finally, 
$\vert \partial_\omega Z_\varepsilon \vert_{r_\varepsilon}$ is bounded as $\varepsilon\rightarrow 0$. In this case, $Z_\varepsilon$ and $\partial_\omega Z_\varepsilon$ have adherent values; let $Z_\infty$ be an adherent value of $Z_\varepsilon$. Since 

$$\partial_\omega (Z_\varepsilon \psi_{\bar{k}})=(A+F)Z_\varepsilon \psi_{\bar{k}} -
Z_\varepsilon \psi_{\bar{k}}(A_\varepsilon + \psi_{\bar{k}}^{-1}\bar{F}_\varepsilon\psi_{\bar{k}})
$$

\noindent
(where $A_\varepsilon\in sl(2,\R)$),
and since all factors except $A_\varepsilon$ are known to converge in a subsequence, then there exists a constant $A_\infty\in sl(2,\R)$ such that

$$\partial_\omega (Z_\infty \psi_{\bar{k}})=(A+F)Z_\infty \psi_{\bar{k}} -
Z_\infty \psi_{\bar{k}}A_\infty 
$$

and thus $A+F$ is actually is a reducible cocycle in $U_{r_\infty}(\T^d, sl(2,\R))$ for some $r_\infty > 0$.

\bigskip
\textbf{Non reducible case}

If the system $A+F$ is not reducible, then for all $k\geq 1$, there exists $k'\geq k$ such that $A_{k'}$ does not have a $BR_\omega^{R_{k'}N_{k'}}(\kappa''(\varepsilon_{k'}))$ spectrum. In this case, $||A_{k'+1}||\leq \kappa''(\varepsilon_{k'})=\kappa \varepsilon_{k'}^\zeta$.

\end{proof}

We will now show a density corollary. 
\begin{corollary}[Density of reducible cocycles close to a constant cocycle]

Let $r_0>0$, $A\in sl(2, \R)$, and $G\in U_{r_0}(2\T^d,sl(2,\R))$ such that $\vert G - A \vert_{r_0} \leq  \varepsilon_0$ and $\Vert A \Vert \leq \varepsilon_0^{-\frac{\zeta}{2}}$ with $\varepsilon_0$ as in \ref{smallness-epsilon}, and satisfying the assumption \ref{assumption-2bis}. Denote

\[\rho = r_0-\frac{150\delta\vert \log \varepsilon_0 \vert}{\pi \Lambda\circ \Psi^{-1}(\epsilon_0^{-\zeta})} -\frac{150\delta}{\pi  \zeta \log(2\delta)}
\int_{\Psi^{-1}(\varepsilon_0^{-\zeta})}^{+\infty}\frac{\Lambda'(t)\ln \Psi(t)}{\Lambda(t)^2} dt.\]

Then for all $\varepsilon > 0$ there exists $H \in U_{\rho}(2\T^d, sl(2,\R))$ such that $\vert G - H \vert_{\rho} \leq \varepsilon$ and $H$ is reducible.
\end{corollary}
\begin{proof}
Apply theorem \ref{theoreme} with $F = G-A$. Since $\rho \leq r_\varepsilon$, we in particular get matrices $Z_\varepsilon \in U_{\rho}(\T^d, SL(2,\R)), \bar A_\varepsilon, \bar F_\varepsilon \in U_{\rho}(\T^d, sl(2,\R))$ and $A_\varepsilon \in sl(2,\R)$ such that
\begin{itemize}
\item $\bar A_\varepsilon$ is reducible to $A_\varepsilon$,
\item $\partial_\omega Z_\varepsilon =(A+(G-A))Z_\varepsilon - Z_\varepsilon(\bar A_\varepsilon + \bar F_\varepsilon) = G Z_\varepsilon - Z_\varepsilon(\bar A_\varepsilon + \bar F_\varepsilon)$,
\item $\vert Z^{\pm 1}_\varepsilon \vert_{\rho} \leq 1 + 
\varepsilon_0^{\frac{9}{10}}\leq 2$,
\item $\vert \bar F_\varepsilon \vert_{\rho} \leq \frac{\varepsilon}{4}$
\end{itemize}
Let $H := G - Z_\varepsilon \bar F_\varepsilon Z_\varepsilon^{-1}$. We have \[\partial_\omega Z_\varepsilon = H Z_\varepsilon - Z_\varepsilon \bar A_\varepsilon \] and then $H$ is reducible to $A_\varepsilon$ (as $\bar A_\varepsilon$ is). Moreover, $H$ satisfies 
\[ \vert H - G \vert_{\rho} = \vert Z_\varepsilon^{-1}\bar F_\varepsilon Z_\varepsilon \vert_{\rho} \leq 4 \vert \bar F_\varepsilon\vert_{\rho} \leq \varepsilon. \]
\end{proof}

\bibliographystyle{elsarticle-num-names}

\end{document}